\documentclass{article}

\usepackage[english]{babel}

\usepackage[letterpaper,top=2cm,bottom=2cm,left=3cm,right=3cm,marginparwidth=1.75cm]{geometry}
\usepackage{amsmath, amsthm, amssymb, appendix, bm, graphicx, hyperref, mathrsfs}
\usepackage{amsmath}
\usepackage{graphicx}
\allowdisplaybreaks[4]
\usepackage{graphicx}
\usepackage{url}
\usepackage{color}
\usepackage{subfigure}
\usepackage[misc]{ifsym}
\usepackage{algorithm}
\usepackage{algorithmic}
\usepackage{bm}
\usepackage{comment}
\usepackage{booktabs}

\newtheorem{theorem}{Theorem}[section]
\newtheorem{definition}[theorem]{Definition}
\newtheorem{proposition}[theorem]{Proposition}
\newtheorem{corollary}[theorem]{Corollary}

\newtheorem{lemma}[theorem]{Lemma}
\newtheorem{remark}[theorem]{Remark}

\newtheorem{example}[theorem]{Example}
\def \alg {\operatorname{ALG}}

\title{Robust Least Squares Problems  with Binary Uncertain Data}

\author{Yang Zhou\thanks{School of Mathematics and Statistics, Shandong Normal University, 250014, Jinan, China (zhouyang@sdnu.edu.cn).  This author's work is supported by Natural Science Foundation of Shandong Province (No. ZR2025MS24) and National Science Foundation of China (No. 12371099).}
              \and
Xiaojun Chen\thanks{Department of Applied Mathematics, The Hong Kong Polytechnic University, Kowloon, Hong Kong, China ({xiaojun.chen@polyu.edu.hk}). This author's work is supported by CAS-Croucher Funding Scheme for the CAS AMSS-PolyU Joint Laboratory in Applied Mathematics and Hong Kong Research Grant Council projects PolyU15300023, PolyU15300024.}
}
             
\begin{document}
\maketitle

\begin{abstract}We propose a Binary Robust Least Squares (BRLS) model that encompasses key robust least squares formulations, such as those involving uncertain binary labels and adversarial noise constrained within a hypercube. {To develop algorithms with theoretical guarantees for the BRLS problem, we exploit the structure of the inner binary maximization problem with a convex quadratic objective function. Refined guarantees are obtained when the noise correlations are sign-structured, in which case the inner problem admits sharper submodular or supermodular oracles.}
{For the supermodular linear BRLS problem, we establish a link between saddle points of its continuous relaxation and global minimax points of BRLS, and propose a projected-gradient algorithm to find an $\epsilon$-global minimax point in $O(\epsilon^{-2})$ iterations. For the supermodular nonlinear BRLS problem, we develop a Moreau-envelope-based framework that finds an $\epsilon$-stationary point in expectation within $O(\epsilon^{-4})$ iterations.}
{For the linear submodular case and the linear general case, we utilize a double-greedy algorithm and a semidefinite relaxation as the respective subsolvers; the latter attains an approximation ratio below $2/\pi$. Coupled with the projected-gradient framework, these oracles yield approximate minimax guarantees within $O(\epsilon^{-2})$ iterations.}
{
Numerical experiments on health status prediction with candidate label-corruption sets, synthetic linear BRLS, and thresholded phase retrieval with missing binary labels illustrate the behavior and robustness gains of the BRLS model under structured noise compared with classical least-squares-based baselines.}

\vspace{0.3cm}

\vspace{0.3cm}
\end{abstract}

\section{Introduction}\label{sec:introduction}
In many practical applications such as robust control, signal denoising, and adversarial machine learning, we are often faced with problems involving uncertain or corrupted measurements, particularly those subject to adversarial perturbations. The \emph{robust least squares (RLS)} problem is designed to enhance robustness against such noise by minimizing the worst-case residual over an uncertainty set \cite{chen2024min,el1997robust}.
In this paper, we consider the following Binary Robust Least Squares (BRLS) problem:
\begin{equation}\label{eq:binary_RLS}
    \min_{x \in \mathcal{X}} \max_{y \in \mathcal Y} \Theta(x,y) := \frac{1}{2} \|F(x) - C y\|^2, \tag{BRLS}
\end{equation}
where $\mathcal{X} \subseteq \mathbb{R}^m$ is a convex compact set, $ F: \mathbb{R}^m \to \mathbb{R}^r $ is a continuous function, $ \mathcal Y :=\{0,1\}^n = \{ y \in \mathbb{R}^n : y_i \in \{0,1\},\; i = 1,\ldots,n\}$, $ C \in \mathbb{R}^{r \times n} $ is a given matrix encoding the structure of noise propagation, and $\|\cdot\|$ denotes the Euclidean norm.
{Unlike models that consider uncertainty in the operator matrix (e.g., total least squares), our focus is on discrete adversarial noise in the observation vector; despite this difference, the BRLS framework captures a broad and important class of problems involving structured data corruption.}

\paragraph{1.}
    When $C = \mathbf{0}$,  problem  (\ref{eq:binary_RLS}) reduces to the classical least squares problem.

\paragraph{2. The robust least squares problem with uncertain binary labels.}
Consider a binary classification problem where the observed labels are given as $ b \in \{0,1\}^r $, and a subset of them, indexed by $ \mathcal{I} \subseteq [r] := \{1,\ldots,r\} $. {The set $\mathcal I$ can be viewed as a candidate unreliable-label set or a missing label set} \cite{Akhtiamov2024}.
To model the uncertainty in labels, we set $n = r$ and define a diagonal matrix $ D \in \mathbb{R}^{r \times r} $, where each diagonal entry is given by
$$
d_i =
\begin{cases}
1 - 2b_i  & \text{if } i \in \mathcal{I}, \\
0 & \text{otherwise}.
\end{cases}
$$
This ensures that when $y$ takes values in $\{0,1\}^n$, the form $b + D y$ can encompass the true observation.
This problem can then be formulated as
\begin{equation}\label{eq:binary_classification}
    \min_{x \in \mathcal{X}} \max_{y \in \mathcal Y} \frac{1}{2} \| \hat{F}(x) - (b + D y) \|^2,
\end{equation}
where $ \hat{F}: \mathbb{R}^m \to \mathbb{R}^r $ is the prediction function.
This model can be rewritten in the form of problem (\ref{eq:binary_RLS})
by setting
$ F(x) = \hat{F}(x) - b $ and
$ C = D $.

{
\paragraph{3. Robust state estimation with structured fault or attack modes.}
Another natural source of BRLS arises in state-estimation and inverse problems where the residual may be affected by a finite collection of structured faults or attacks. For example, in power-system state estimation, least-squares and weighted least-squares formulations are standard tools for estimating system states from meter measurements \cite{abur2004power}. False-data injection attacks can manipulate selected measurements and bias the estimated state while evading residual-based bad-data detection \cite{liu2009false}. If the $i$th candidate fault or attack template contributes a vector $c_i$ to the residual, then the binary variable $y_i$ indicates whether this mode is activated and the aggregate perturbation is $Cy=\sum_i y_i c_i$. This leads directly to problem (\ref{eq:binary_RLS})
where $F(x)=Hx-b$ with a measurement matrix $H \in \mathbb{R}^{r\times m} $ and an observed measurement vector $b\in \mathbb{R}^r $. In this problem, the columns of $C$ are not artificial labels but physically or operationally meaningful disturbance templates.
}

\paragraph{4. The robust least squares problem with hypercube-constrained {noise}.}
A typical robust least squares problem considers hypercube-constrained noise variables in the following form:
\begin{equation}\label{eq:binary_RLS_chen}
    \min_{x \in \mathcal{X}} \max_{z \in [-\delta,\delta]^n} \frac{1}{2} \|\hat F(x) - \hat C z\|^2, \tag{HRLS}
\end{equation}
where $[-\delta,\delta]^n := \{z : \|z\|_\infty \le \delta\} \subseteq \mathbb R^n$ with a given scalar $\delta > 0$.
This model has been studied by {\cite{chen2024min}}.
{By performing the change of variables,
we show that problems (BRLS) and (HRLS) have the same inner value function in $x$. Moreover, if $C$ has  full column rank, then problems (BRLS) and (HRLS) have the same set of global minimax points.}
The formal derivation of this equivalence and the structural properties of the corresponding inner value function are shown in Subsection~\ref{sec:connection}.

\vspace{0.5cm}

The analysis of minimax problems often relies on a clear understanding of their  saddle points \cite{v1928theorie}.
\begin{definition}\label{def:saddle_point}
    A pair $(x^*,y^*) \in \mathcal X \times \mathcal Y$ is said to be a saddle point of \ref{eq:binary_RLS} if
    \begin{equation*}
        \Theta(x^*,y)\le \Theta(x^*,y^*) \le \Theta(x,y^*),
    \end{equation*}
    for any $x \in \mathcal X$ and $y \in \mathcal Y$.
\end{definition}

It has been shown that even when \( F \) is affine, problem (\ref{eq:binary_RLS_chen}) may not admit saddle points \cite{chen2024min}. Furthermore, for a fixed \( x \in \mathcal{X} \), solving the inner maximization problem to optimality possibly remains NP-hard \cite{ye1992affine}. These challenges motivate the analysis of approximate minimax points, defined as follows. This concept {first} appeared in \cite{adibi2022minimax}, and our definition extends it by imposing approximation conditions on both variables $ x $ and $ y $.

\begin{definition} \label{def:approx_minimax}
A pair $ (x^*, y^*) $ is called an $(\alpha, \epsilon)$-approximate minimax point of \ref{eq:binary_RLS} if it satisfies
\begin{equation}\label{eq:def_appro_minimax_point}
  \alpha \max_{y \in \mathcal{Y}} \Theta(x^*, y)  \le \Theta(x^*, y^*) \le \frac 1 \alpha \min_{x \in \mathcal{X}} \max_{y \in \mathcal{Y}} \Theta(x, y) + \epsilon,
\end{equation}
where $ \alpha \in (0,1] $ and $\epsilon \ge 0$.
\end{definition}

If (\ref{eq:def_appro_minimax_point}) holds with
$\alpha=1$, then $(x^*,y^*)$ is called an $\epsilon$-global minimax point. If (\ref{eq:def_appro_minimax_point}) holds with
$\alpha=1$ and $\epsilon=0$, then $(x^*,y^*)$ is called a global minimax point  \cite{jin2020local}.
Since \( \mathcal{X} \) is compact and convex, \( \mathcal{Y} \) is finite, and \( F \) is continuous, the existence of a global minimax point for \ref{eq:binary_RLS} is guaranteed. As every global minimax point trivially satisfies the definition of an \( (\alpha, \varepsilon) \)-approximate minimax point for any \( \alpha \in (0,1] \) and \( \epsilon \ge 0 \), it follows that such approximate minimax point also exists.

The main challenge in solving \ref{eq:binary_RLS} arises from its mixed discrete-continuous structure.
{Our key insight is to address this by designing provable subroutines for the inner maximization problem. For any fixed outer variable $x \in \mathcal X$, the inner problem is an unconstrained Binary Quadratic Programming (BQP). Further, the structural properties of the noise propagation matrix $ C $, specifically whether its columns form acute or obtuse angles, determine the modularity of the objective function with respect to $ y $. This connection allows us to leverage powerful tools from BQP, submodular, and supermodular optimization to design algorithms with theoretical guarantees. For a general matrix $ C $, we employ a classical SDP relaxation combined with deterministic rounding to construct the subproblem solver.
When all pairwise inner products of the columns of $C$ are nonnegative, the inner problem becomes supermodular and admits polynomial-time solutions;
when all pairwise inner products are nonpositive, it becomes submodular and can be solved greedily with provable approximation guarantees. This leads to algorithmic frameworks that exploit the discrete structure of the uncertainty set.}

{
	The focus of this paper is a robust least-squares template whose inner adversarial problem has an explicit binary and combinatorial structure, rather than a task-specific procedure for any single application domain. This perspective allows us to connect robust least squares with tools from submodular optimization, supermodular optimization, and semidefinite relaxations, and to obtain provable approximate minimax guarantees for several structured regimes.}

It has been shown that RLS problems under ellipsoidal uncertainty sets can be solved efficiently via convex reformulations, allowing for optimal solutions in polynomial time \cite{el1997robust}.
However, the RLS with hypercube uncertainty, i.e., the \ref{eq:binary_RLS_chen} problem is generally NP-hard to solve \cite{murty1985some,ye1992affine}. Recent work \cite{chen2024min} tackles this by leveraging a special structure: when the noise propagation matrix $C$ has a QR decomposition with a positive definite  diagonal matrix  \( R \), the  value function of the inner maximization problem has an explicit formula.
For the least squares formulation with discrete adversarial perturbations, while regularization-based approaches have been studied in the context of binary classification with noisy labels \cite{Akhtiamov2024}, the robust model discussed in this paper has not been systematically addressed in the literature.

We incorporate submodular and supermodular maximization solvers that are well established in the combinatorial optimization literature, using them as subsolvers within our algorithmic framework.
Submodular minimization can be solved in polynomial time and has been extensively studied through a variety of approaches. A detailed overview of these methods is given in \cite[Section 2.3.3]{bilmes2022submodularity}. Numerically, the algorithms by \cite{fujishige2005submodular} and \cite{wolfe1976finding} exhibit strong performance, with the best strongly polynomial algorithms requiring nearly $ O(n^3) $ calls to the function-evaluation oracle \cite{jiang2022minimizing}.
In contrast, submodular maximization is generally NP-hard, but efficient approximations exist.  Buchbinder et al. \cite{buchbinder2015tight} proposed a deterministic double greedy algorithm with a $ \frac{1}{3} $-approximation guarantee for unconstrained maximization over $ \{0,1\}^n $, which is particularly relevant to our work. Extensions to more complex constraints have been studied by Nemhauser et al. and others (see~\cite{calinescu2011maximizing,chekuri2014submodular,nemhauser1978analysis} and references therein).

The main contributions of this paper are as follows.
{
\begin{enumerate}
\item We propose the BRLS model for robust least-squares problems with binary perturbation patterns. We show its efficiency for handling binary label uncertainty/missing and clarify its connection with hypercube-constrained robust least squares through an extreme-point reformulation.
\item In the supermodular linear case, by introducing the Lov\'asz extension of~\ref{eq:binary_RLS}, namely, the \ref{eq:minimax_lovasz} model, {we establish a link between saddle points of \ref{eq:minimax_lovasz} and global minimax points of~\ref{eq:binary_RLS}: the $x$-component of any saddle point of \ref{eq:minimax_lovasz}, together with a binary inner maximizer, yields a global minimax point of BRLS.}
    We then propose a projected gradient algorithm with a Lov\'asz-extension-based exact inner oracle \cite{jiang2022minimizing} as a subsolver  that computes an $\epsilon$-global minimax point in $O(\epsilon^{-2})$ iterations.
For the supermodular nonlinear case, we propose a randomized projected gradient framework that finds an $\epsilon$-stationary point in expectation within $O(\epsilon^{-4})$ iterations, where stationarity is measured by the gradient norm of the Moreau envelope.

\item In the non-supermodular linear case, we distinguish between the submodular setting and the general case. For these two scenarios, we propose a double greedy algorithm and an SDP-based oracle as the respective subsolvers.     When integrated into the projected-gradient framework for linear BRLS, these subsolvers yield $(1/3,\epsilon)$-approximate minimax points in the submodular case and $(\frac2 \pi - \eta , \frac \epsilon {\frac 2 \pi -\eta})$-approximate minimax points in the general case, for any fixed $\eta\in(0,2/\pi)$, within $O(\epsilon^{-2})$ outer iterations.
\end{enumerate}
}

The rest of the paper is organized as follows. Section~\ref{sec:reformulation} formally introduces the relationship between \ref{eq:binary_RLS} and \ref{eq:binary_RLS_chen}, and analyzes the structural properties of \ref{eq:binary_RLS}. Section~\ref{sec:supermodular} investigates the supermodular case, establishing theoretical connections between \ref{eq:binary_RLS} and its continuous formulations, and develops algorithmic frameworks for {both the linear and nonlinear cases of $F$}. Section~\ref{sec:non-supermodular} studies non-supermodular BRLS: it first treats the submodular linear case using a double-greedy approximation oracle, and then handles general linear noise-correlation matrices through an SDP-based approximation oracle. Numerical experiments are reported in Section~\ref{sec:numerical}. In Section \ref{sec:conclusions} we give the conclusion of this paper.

\section*{Notation}

The symbols $\mathbf{1}$ and $\mathbf{0}$ denote vectors or matrices of all ones and all zeros respectively with dimensions inferred from context. The symbol $\mathbf{e}_i$ represents the $i$-th column of the identity matrix in $\mathbb{R}^{n \times n}$. The notation $[n]$ stands for the integer set $\{1, 2, \dots, n\}$. The set $\operatorname{SOL}(f,\mathcal{C})$ denotes the optimal solution set to the optimization problem  $\max_{x \in \mathcal{C}} f(x)$.
The expression $\nabla_x f(x,y)$ denotes the gradient of $f$ with respect to $x$. The projection $\operatorname{Proj}_{\mathcal{X}}(x)$ is the Euclidean projection of point $x$ onto the closed convex set $\mathcal{X}$ defined as $\operatorname{Proj}_{\mathcal{X}}(x) = \arg\min_{z \in \mathcal{X}} \|z - x\|^2$.
For a vector \( y\in \mathbb{R}^n \), we denote by \( \operatorname{supp}(y) = \{i :  y_i\neq 0, i\in [n]\} \).
For two vectors \( y, y'\in \mathbb{R}^n \), the symbols \( y \vee y' \) and \( y \wedge y' \) denote the componentwise maximum and minimum, respectively, defined by \( (y \vee y')_i := \max\{y_i, y'_i\} \) and \( (y \wedge y')_i := \min\{y_i, y'_i\} \) for each \( i \in [n] \).

\section{Structural Properties of the Inner Objective under Noise Correlation}\label{sec:reformulation}
In this section, we focus on the structural properties of the inner maximization problem in \ref{eq:binary_RLS}. Our goal is to understand how the correlation structure of the noise is determined by the geometry of the matrix $ C $, which allows us to leverage powerful tools from  submodular optimization for solving the inner problem efficiently. Before that, we begin with a reformulation of \ref{eq:binary_RLS_chen} which shows that it can be expressed using binary variables.

\subsection{Connection between \ref{eq:binary_RLS} and \ref{eq:binary_RLS_chen}}\label{sec:connection}

In this subsection, we show that finding a global minimax point of \ref{eq:binary_RLS_chen} can be achieved by finding one of \ref{eq:binary_RLS}.
Firstly, we perform the following change of variables:
$$
y = \frac{z}{2\delta} + \frac{1}{2} \mathbf{1}.
$$
This transformation maps $ z \in [-\delta,\delta]^n $ bijectively onto $ y \in [0,1]^n $.
 By letting $ F(x) = \hat F(x) + \delta \hat C \mathbf{1} $ and $ C = 2\delta \hat C $, the \ref{eq:binary_RLS_chen} problem can be reformulated to
\begin{equation}\label{eq:binary_RLS2_chen}
    \min_{x \in \mathcal{X}} \max_{y \in [0,1]^n} \Theta(x,y):= \frac{1}{2} \|F(x) - C y\|^2. \tag{R-HRLS}
\end{equation}

Denote the inner value functions of \ref{eq:binary_RLS} and \ref{eq:binary_RLS2_chen} by
$$
\varphi(x) := \max_{y \in \{0,1\}^n }  \Theta(x,y) , {\rm ~~ and ~~} \hat \varphi(x) := \max_{y \in [0,1]^n }  \Theta(x,y).
$$
Let $ \operatorname{SOL}(\Theta(x,\cdot), \{0,1\}^n) $ and $ \operatorname{SOL}(\Theta(x,\cdot), [0,1]^n) $ denote the sets of optimal solutions to the above two maximization problems, respectively. We now show the relationship between these two sets.

\begin{proposition}\label{prop:P_x_discrete}
For any fixed $ x \in \mathcal{X} $, it holds that $ \operatorname{SOL}(\Theta(x,\cdot), \{0,1\}^n)  \subseteq  \operatorname{SOL}(\Theta(x,\cdot), [0,1]^n)$. If $ C $ has full column rank, then the two sets are equal.
\end{proposition}

\begin{proof}
The feasible set $ \{0,1\}^n $ consists of all extreme points of the hypercube $ [0,1]^n $. Since the function $ \Theta $ is convex in $ y $ for any $x$, its maximum over a convex polytope must be attained at one of the extreme points. Therefore, we have $ \operatorname{SOL}(\Theta(x,\cdot), \{0,1\}^n)  \subseteq  \operatorname{SOL}(\Theta(x,\cdot), [0,1]^n)$.

Now suppose that the matrix $ C \in \mathbb{R}^{r \times n} $ has full column rank. In this case, the function $ \Theta(x, \cdot) $ is strictly convex in $ y $.  Then for any ${y} \in [0,1]^n \setminus\{0,1\}^n $ {we can write $y$ as a convex combination of the $2^n$ vertices of the hypercube:}
{$$y = \sum_{\ell=1}^{2^n} \lambda_\ell y^\ell, \text{~~with~}\ \sum_{\ell=1}^{2^n} \lambda_\ell = 1,  ~\lambda_\ell \ge 0, ~y^\ell \in \{0,1\}^n, ~~ \ell = 1,\ldots,2^n. $$
Since $y\notin\{0,1\}^n$, at least two coefficients $\lambda_\ell$ are positive. Thus, for any  $y^* \in \operatorname{SOL}(\Theta(x,\cdot), \{0,1\}^n) $ we have}
$$
{\Theta(x, {y}) < \sum_{\ell=1}^{2^n} \lambda_\ell \Theta(x, y^\ell) \leq \Theta(x, y^*),}
$$
and hence $ \operatorname{SOL}(\Theta(x,\cdot), \{0,1\}^n)  = \operatorname{SOL}(\Theta(x,\cdot), [0,1]^n)$.
This completes the proof.
\end{proof}

From this result, we immediately obtain the equivalence between the two value functions.

\begin{corollary}\label{coro:phi_equiv}
{For any given $ x \in \mathcal{X} $, it holds that $ \varphi(x) = \hat{\varphi}(x) $. Thus, the two formulations \ref{eq:binary_RLS} and \ref{eq:binary_RLS2_chen} have the same set of global minimizers in the variable $x$. If, in addition, $C$ has full column rank, then the two formulations have the same set of global minimax pairs.}
\end{corollary}

{
\begin{remark}
The full-column-rank condition on $C\in\mathbb R^{r\times n}$ is used to guarantee equality between the inner maximizer sets over $\{0,1\}^n$ and $[0,1]^n$, which is needed for the minimax-pair equivalence in Corollary~\ref{coro:phi_equiv}. It is automatically satisfied when the noise propagation directions $c_1,\ldots,c_n$ are linearly independent; in particular, the column-orthogonal case satisfies this condition. When $C$ is rank deficient, the value-function equivalence $\varphi(x)=\hat\varphi(x)$ still holds. What may fail is the equality of the inner maximizer sets, because fractional maximizers can occur in the continuous formulation.
\end{remark}
}

\subsection{Supermodularity or Submodularity under Noise Correlations}

We now examine the structural properties of the function $\Theta$ with respect to the variable $y$, particularly under varying configurations of the matrix $C$.
Let
$$
C = (c_1, \ldots, c_n) \in \mathbb{R}^{r \times n}.
$$
Without loss of generality, we assume $c_i \neq \mathbf 0$ for all $i \in [n]$, and define
$$
\theta_{ij} = \arccos\frac{c_i ^\top c_j}{\|c_i\| \|c_j\|}, \quad \forall ~i,j \in [n].
$$

In \cite{chen2024min}, it was shown that the inner value function in problem (\ref{eq:binary_RLS_chen}) admits an explicit form
$$\max_{y\in [-\delta,\delta]^n}\frac{1}{2}\|\hat{F}(x)-\hat{C}y\|^2
=
\frac{1}{2}\|\hat{F}(x)\|^2
+\delta \|\hat{C}^\top\hat{F}(x)\|_1
+\frac{1}{2}\|\hat{C}\|_F^2\delta^2$$
if
the thin QR factorization $\hat C = QR$ has $Q^\top Q=I$ and $R$ being diagonal with strictly positive diagonal entries.
This implies that $\theta_{ij}=\pi/2$ for all $i\ne j.$

{
When
$\theta_{ij}\leq \pi/2$ for all $i\ne j$, the propagated noise directions
are pairwise nonnegatively correlated. This is consistent with common-source
or common-mode perturbations, where several channels are affected by the same
external interference or environmental factor. For example, power-line or
electromagnetic interference in multi-channel measurement systems may enter
several channels in the same direction, producing reinforcing perturbation
components \cite{huhta1973interference,ott2009emc}. In this case, selecting
multiple binary noise components tends to amplify the aggregate perturbation
$Cy$.

When $\theta_{ij}\geq \pi/2$ for all $i\ne j$, the propagated noise directions
are pairwise nonpositively correlated. This condition is more restrictive, but
it captures idealized competitive or fixed-total perturbation models. In
multinomial, compositional, or resource-allocation data, the components are
constrained by a fixed total mass; increasing one component necessarily reduces
the remaining mass available to the others, which induces negative pairwise
dependence among components \cite{aitchison1982statistical,aitchison1986statistical}.
Thus the obtuse case can be viewed as a model for antagonistic noise sources
whose effects compete rather than reinforce each other.

The orthogonal case $\theta_{ij}=\pi/2$ for all $i\ne j$ corresponds to
decoupled perturbation channels. Each binary variable then controls one
independent component of $Cy$, as in the label-flipping model
\eqref{eq:binary_classification}. Similar orthogonal decompositions appear in
communication systems such as OFDM, where information is transmitted over
mutually orthogonal subcarriers to reduce inter-channel interference
\cite{vannee2000ofdm,weinstein1971data}.
}

{For brevity, we give the following geometric shorthand for the above sign-structured regimes.}

\begin{definition}[Acute/Obtuse Matrix] \label{def:acute_obtuse}
A matrix $ C \in \mathbb{R}^{r \times n} $ is said to be \emph{acute} if $ \theta_{ij} \leq \frac{\pi}{2} $, $i \neq j$, and \emph{obtuse} if $ \theta_{ij} \geq \frac{\pi}{2} $, $i \neq j$, {for all} $i,j \in [n]$.
\end{definition}

A matrix $ C $ is column-orthogonal if and only if it is both acute and obtuse. It is worth noting that these geometric characterizations of $ C $ lead to the supermodularity or submodularity of  $ \Theta(x, \cdot) $.

\begin{definition} \label{def:submodular}
A function $ h: \{0, 1\}^n \to \mathbb{R} $ is called \emph{submodular} if for any $ y , y' \in \{0,1\}^n $ with $y \le y'$ and any index $ i \in [n] \setminus \operatorname{supp}(y') $, it holds that
$$
h(y + \mathbf{e}_i) - h(y) \geq h(y' + \mathbf{e}_i) - h(y').
$$
A function $ h $ is called \emph{supermodular} if $ -h $ is submodular, and \emph{modular} if it is both submodular and supermodular.
\end{definition}

An equivalent characterization of submodularity on the set $ \{0,1\}^n $ is that {\cite{nemhauser1978analysis}}
$$
h(y) + h(y') \geq h(y \vee y') + h(y \wedge y'), \quad \forall y, y' \in \{0,1\}^n.
$$

Now we can establish the relationship between the geometry of the matrix $ C $ and the submodular or supermodular structure of the inner objective function.

\begin{proposition}\label{prop:sub_sup}
The function $ \Theta (x,\cdot)$ is supermodular (resp., submodular) for any fixed $ x \in \mathcal{X} $ if and only if the matrix $ C $ is acute (resp., obtuse).
\end{proposition}

\begin{proof}
We begin by expanding the function $ \Theta $ as follows:
\begin{equation}
\Theta(x, y) =
\frac{1}{2} \|F(x)\|^2 - F(x)^{\top} C y + \frac{1}{2} y^{\top} C^{\top} C y.
\end{equation}

For fixed $ x \in \mathcal{X} $, define $ u := C^\top F(x) $ and $ \Sigma := C^\top C $. Then we have
\begin{equation}
\Theta(x, y) = \frac{1}{2} y^\top \Sigma y - u^\top y + \frac{1}{2} \|F(x)\|^2
\end{equation}
and
\begin{eqnarray}
\Theta(x, y + \mathbf{e}_i) - \Theta(x, y)
& = & y^\top \Sigma \mathbf{e}_i + \frac{1}{2} \mathbf{e}_i^\top \Sigma \mathbf{e}_i - u^\top \mathbf{e}_i.
\end{eqnarray}
Hence for any $y, y' \in \{0,1\}^n$, $ y \le y' $  and $ i \in [n] \setminus \operatorname{supp}(y') $,
\begin{eqnarray}
&& \Theta(x, y' + \mathbf{e}_i) - \Theta(x, y') - \left[ \Theta(x, y + \mathbf{e}_i) - \Theta(x, y) \right] \nonumber \\
& = & (y' - y)^\top \Sigma \mathbf{e}_i \nonumber \\
& = & \sum_{j: y'_j > y_j}\mathbf{e}_j^\top \Sigma \mathbf{e}_i \nonumber \\
& = & \sum_{j: y'_j > y_j} c_i^\top c_j.
\end{eqnarray}

Based on this expression, we can now proceed to prove the equivalence.

\textbf{Sufficiency.}
If $ C $ is acute (i.e., $ c_i^\top c_j \ge 0 $ for all $ i \ne j $), then the above sum is nonnegative, implying that $ \Theta(x, \cdot) $ is supermodular. Similarly, if $ C $ is obtuse, then the sum is nonpositive, implying submodularity.

\textbf{Necessity.}
Suppose $ \Theta(x, \cdot) $ is supermodular for any $ x \in \mathcal{X} $. Consider the case where $ y = \mathbf{0} $ and $ y' = \mathbf e_j $ for some $ j \ne i $. Then the above difference reduces to
$$
\Theta(x, y' + \mathbf{e}_i) - \Theta(x, y') - \left[ \Theta(x, y + \mathbf{e}_i) - \Theta(x, y) \right] = c_i^\top c_j.
$$
Supermodularity implies that this quantity must be nonnegative for all $ i \ne j $, hence $ C $ must be acute. The same logic applies to submodularity and obtuse matrices.

This completes the proof.
\end{proof}

It is evident from the setting described in \cite{chen2024min} that the function $ \Theta(x, \cdot) $ is modular for any fixed $ x \in \mathcal X $. By the well-known equivalence between modular functions and separable functions \cite{topkis1978minimizing}, this implies that, for any $ x $, solving the inner maximization of \ref{eq:binary_RLS}  reduces to optimizing each coordinate of $ y $ independently. This observation aligns with the approach taken in \cite{chen2024min}, where an explicit solution to the inner maximization problem is derived by separately analyzing each component of $ y $.

\section{Supermodular \ref{eq:binary_RLS}}
\label{sec:supermodular}

In this section, we develop algorithmic frameworks for solving the \ref{eq:binary_RLS} problem under the assumption that the inner objective function is supermodular in the noise variable $ y $. This structural property, established in the previous section, enables us to exploit powerful tools from supermodular optimization to design efficient algorithms with theoretical guarantees.
We divide the discussion into two subsections: the case where $ F $ is affine, which corresponds to a linear RLS formulation that is convex in $x$, and the case where $ F $ is  differentiable, in which $ \Theta $ fails to preserve convexity in $ x $.

\subsection{Supermodular Linear \ref{eq:binary_RLS}}

In this setting,  we first introduce the \emph{Lov\'asz extension} to better understand the structure of the inner problem, which allows us to lift a discrete function into a continuous domain while preserving key structural properties.

\begin{definition}[Lov\'asz Extension, \cite{grotschel2012geometric}]\label{def:lovasz_extension}
Given a function $ h: \{0,1\}^n \to \mathbb{R} $, the \emph{Lov\'asz extension} $ h^L: [0,1]^n \to \mathbb{R} $ of $ h $ is defined as follows. For any $ y \in [0,1]^n $, sort its components in decreasing order: $ y_{j_1} \geq y_{j_2} \geq \cdots \geq y_{j_n} $, where $ (j_1, j_2, \ldots, j_n) $ is a permutation of $ \{1, 2, \ldots, n\} $. Let $ y_{j_0} = 1 $ and $ y_{j_{n+1}} = 0 $. Then, the Lov\'asz extension $ h^L(y) $ is given by
\begin{eqnarray}
    h^L(y)  & = & h(\mathbf 0) + \sum_{k=1}^{n} y_{j_k} \left[ h(\mathbf e_{\{j_1, \ldots, j_k\}}) - h(\mathbf e_{\{j_1, \ldots, j_{k-1}\}}) \right]\\
  \label{eq:lovasz2}  & = & \sum_{k=0}^{n} h(\mathbf e_{\{j_1, \ldots, j_k\}}) (y_{j_k} - y_{j_{k+1}}),
\end{eqnarray}
where $\mathbf e_{S}$ denotes a binary vector with 1 at the support set $S \subseteq [n]$ and 0 otherwise.
\end{definition}

When $ h $ is submodular, a key relationship exists between the minimizers of $ h $ and those of its Lov\'asz extension, as summarized in the following proposition in \cite[Propositions 3.6 and 3.7]{bach2013learning}.

\begin{proposition}[\cite{bach2013learning}]\label{prop:lovasz extension}
For a function $h$, let its Lov\'asz extension $h^L$ be defined as in Definition \ref{def:lovasz_extension}. Then, $ h $ is submodular if and only if $ h^L $ is convex. Furthermore, if $ h $ is submodular, the set of minimizers of $ h^L $ over $ [0,1]^n $ is the convex hull of the minimizers of $ h $ over $ \{0,1\}^n $.
\end{proposition}

For a fixed $x$, let  $\Theta^L(x,\cdot)$  be the Lov\'asz extension of $\Theta(x,\cdot)$. Observe that if \( \Theta(x, \cdot) \) is supermodular, then \( -\Theta(x, \cdot) \) is submodular.
By Proposition \ref{prop:lovasz extension}, since \(-\Theta^L(x,\cdot)\) is the Lov\'asz extension of \(-\Theta(x,\cdot)\), \(\Theta^L(x,\cdot)\) is concave.
This allows us to reformulate the original discrete inner maximization problem over $ y \in \{0,1\}^n $ into a continuous optimization problem over $ y \in [0,1]^n $, which facilitates the use of convex-concave minimax optimization tools:
\begin{equation}\label{eq:minimax_lovasz}
\min_{x \in \mathcal{X}} \max_{y \in [0,1]^n} \Theta^L(x, y).\tag{L-BRLS}
\end{equation}
The following proposition establishes the existence of a saddle point for \ref{eq:minimax_lovasz} and its connection to the global minimax points of \ref{eq:binary_RLS} and~\ref{eq:binary_RLS2_chen}.

\begin{proposition}\label{prop:lovasz_minimax}
Assume that $ C $ is acute and $ F $ is affine. Then \ref{eq:minimax_lovasz} admits a saddle point $ (x^*, y^*) $, and for every $ \bar{y} \in \operatorname{SOL}(\Theta(x^*,\cdot) , \mathcal Y) $, the pair $ (x^*, \bar{y}) $ is a global minimax point of \ref{eq:binary_RLS}. Moreover, for any saddle point $(x^*,y^*)$ of \ref{eq:minimax_lovasz}, $ y^* $ lies in the convex hull of $ \operatorname{SOL}(\Theta(x^*,\cdot), \mathcal Y) $.
\end{proposition}

\begin{proof}
In this setting, Proposition~\ref{prop:lovasz extension} implies that $ \Theta^L(x, y) $ is concave in $ y $ for any $x \in \mathcal X$.
Furthermore, since $ F $ is affine in $ x $, it follows that $ \Theta(x, y) $ is convex in $ x $ for any $y \in \mathcal Y$. By Equation (\ref{eq:lovasz2}), for any fixed $ y \in [0,1]^n $, function $ \Theta^L(x, y) $ is a nonnegative linear combination of $ \Theta(x, y') $ with all $y' \in \mathcal Y$, and therefore remains convex in $ x $.
Now we consider the minimax problem \ref{eq:minimax_lovasz}.
Since $ \mathcal{X} $ is compact and convex, the Sion Theorem \cite{Sion}  guarantees the existence of a saddle point $ (x^*, y^*) $ of the \ref{eq:minimax_lovasz} problem.

Let the inner value function of \ref{eq:minimax_lovasz} be $\varphi^L(x) := \max_{y \in [0,1]^n} \Theta^{L}(x,y)$.
According to \eqref{eq:lovasz2}, for any $x \in \mathcal{X}$ and $y \in [0,1]^n$, $\Theta^{L}(x,y)$ is a convex combination of the values $\Theta(x, y')$ over all $ y' \in \{0,1\}^n$.
Then the inner maximization of $\varphi^L(x)$ can be attained at some point in $\mathcal Y$ and hence $\varphi(x) = \varphi^L(x)$ for all $x \in \mathcal{X}$.
Therefore, $x^*$ minimizes both of them, and the pair $(x^*, \bar y)$ defined in the proposition is a global minimax point of \ref{eq:binary_RLS}.

Finally, by Proposition~\ref{prop:lovasz extension} and the fact that $y^*$ maximizes $\Theta^L(x^*,\cdot)$ on $[0,1]^n$ we complete the proof.
\end{proof}

To illustrate the relationship between \ref{eq:minimax_lovasz} and~\ref{eq:binary_RLS2_chen} established in Proposition~\ref{prop:lovasz_minimax}, we consider the following simple example.

\begin{example}
Consider the \ref{eq:binary_RLS_chen} problem with
$$
\Theta(x, y) = (x - y)^2, \quad x, y \in [-1, 1].
$$
To simplify the form without affecting its essential structure, here the domain of $y$ has been simply shifted and scaled from $[0,1]$ to [-1,1]. By restricting $ y $ to $\{-1, 1\}$ and taking the Lov\'asz extension, we obtain
$$
\Theta^L(x, y) = \frac{y + 1}{2}(x - 1)^2 + \frac{1 - y}{2}(x + 1)^2.
$$
The graphs of these two functions are depicted in Figure \ref{fig:minimax_saddle}. One can verify that $(0, 0)$ is a saddle point of \ref{eq:minimax_lovasz}, while both $(0, 1)$ and $(0, -1)$ are global minimax points of \ref{eq:binary_RLS} and \ref{eq:binary_RLS_chen}. Obviously, $(0, 0)$ lies precisely in the convex hull of $(0, 1)$ and $(0, -1)$.
\end{example}

\begin{figure}[htbp]
	\centering
	\includegraphics[width=1\textwidth]{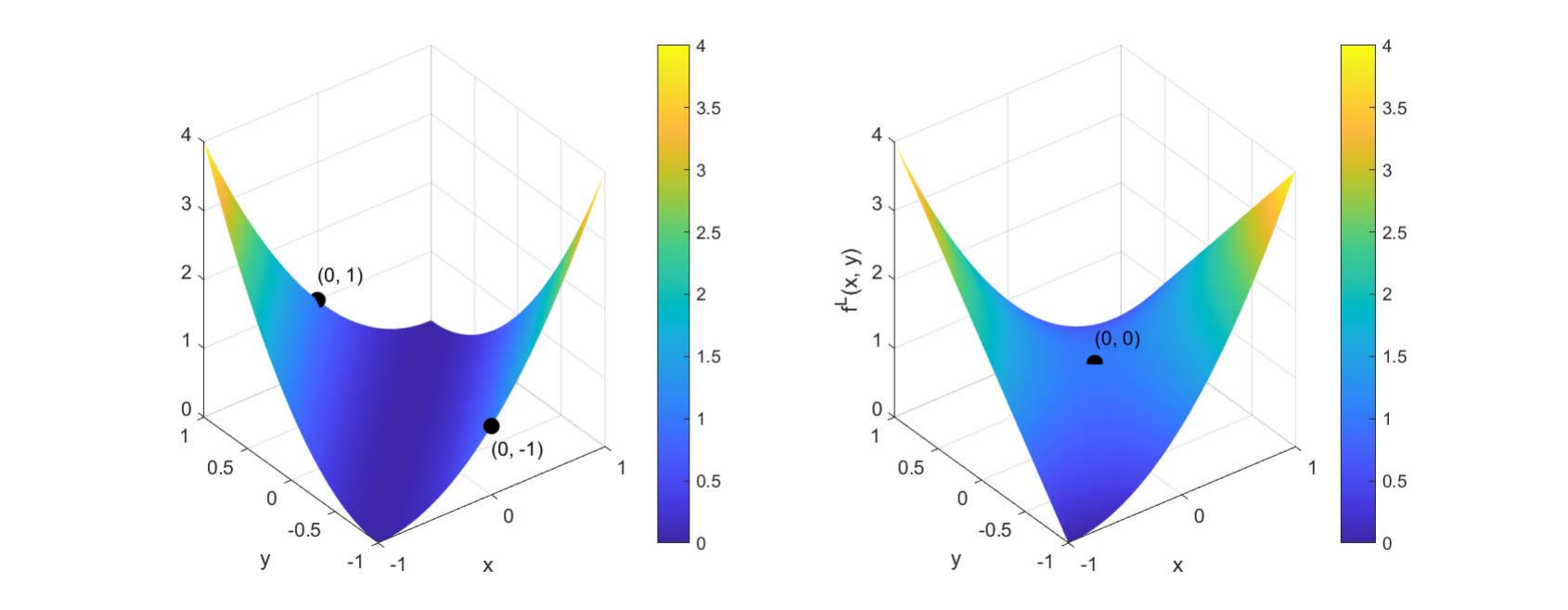}
	\caption{The function $\Theta = (x-y)^2 $ in \ref{eq:binary_RLS_chen} with its global minimax points (left) and $\Theta^L$ in \ref{eq:minimax_lovasz} with its saddle point (right).}
	\label{fig:minimax_saddle}
\end{figure}

The \ref{eq:binary_RLS2_chen} problem is not a convex-concave minimax problem and gradient-based methods cannot ensure to find a global minimax point of the problem from any initial point. However, based on Proposition~\ref{prop:lovasz_minimax}, we can devise an ideal two-step strategy to find a global minimax point via its relationship with \ref{eq:minimax_lovasz}:
\begin{enumerate}
    \item Compute a saddle point $(\hat x, \hat y)$ of the convex-concave minimax problem~\ref{eq:minimax_lovasz}, for instance using a subgradient descent-ascent method.
    \item Fix $ x = \hat x $ and solve the inner maximization problem in $\varphi(\hat x)$ to obtain a global minimax point of the original problem~\ref{eq:binary_RLS}.
\end{enumerate}

In particular, we present Algorithm \ref{alg:RLLS_submax}, a projected gradient-type algorithm for \ref{eq:binary_RLS} with an affine function $F$.
This algorithm relies on a solver $\alg(\Theta(x,\cdot))$ called $\gamma$-approximation algorithm.  In the following, we give a formal definition of $\gamma$-approximation algorithm, which is commonly studied in combinatorial optimization \cite[Appendix A.3]{vazirani2001approximation}.

    \begin{algorithm}[h]
\caption{Projected Gradient Algorithm for Linear \ref{eq:binary_RLS}}
\label{alg:RLLS_submax}
\begin{algorithmic}[1]
\REQUIRE $\Theta$, $ x_0 \in \mathcal{X} $,  $ K $.
\FOR{$ k = 0, \ldots, K-1 $}
    \STATE $ y_k \leftarrow \alg(\Theta(x_k, \cdot)) $;
    \STATE $ x_{k+1} \leftarrow \operatorname{Proj}_{\mathcal{X}}\left(x_k - K^{-\frac{1}{2}} \nabla_x \Theta(x_k, y_k)\right) $;
\ENDFOR
\RETURN $ \hat{x} = \frac{1}{K} \sum_{k=0}^{K-1} x_k $ and $ \hat{y} = \alg(\Theta(\hat x, \cdot)) $.
\end{algorithmic}
\end{algorithm}

\begin{definition} \label{def:approx_alg}
An algorithm is said to be a $\gamma$-approximation algorithm for a maximization problem
\begin{equation}\label{eq:def_approximation_alg}
  \max_{y\in \mathcal Y} h(y)
\end{equation}
with non-negative objective $ h $, if on each instance it returns a feasible solution $ \hat{y} \in \mathcal Y $ in polynomial time such that
$$
h(\hat{y}) \ge \gamma  h(y^*),
$$
where $ y^* $ is an optimal solution of (\ref{eq:def_approximation_alg}) and $ \gamma \in (0,1] $.
\end{definition}

Since $\mathcal{X}$ is compact and convex, we denote by $D > 0$ its diameter, i.e., $\|x - x'\| \leq D$ for all $x, x' \in \mathcal{X}$. When $F$ is affine, the function $\Theta(x, y)$ is globally Lipschitz continuous in $x$ over $\mathcal{X}$ for any fixed $y \in \{0,1\}^n$. In this case, we let $L > 0$ denote a common Lipschitz constant valid uniformly over all such $y$.
The following lemma provides a formal guarantee on the quality of the solution returned by the projected gradient algorithm.

\begin{lemma}\label{thm:alg_PG}
Assume that $F$ is affine, and let $K = \left\lceil \left( \frac{D^2 + L^2}{2\epsilon} \right)^2 \right\rceil$ for a given $\epsilon > 0$. If the subsolver $\alg$ is a $\gamma$-approximation algorithm for maximizing $\Theta(x, \cdot)$ over $\{0,1\}^n$ for any $x \in \mathcal{X}$, then the output $(\hat{x}, \hat{y})$ of Algorithm~\ref{alg:RLLS_submax}
is a $(\gamma, \epsilon/\gamma)$-approximate minimax point of \ref{eq:binary_RLS}.
\end{lemma}

\begin{proof}
Since $ F$ is affine in $ x $, the function $ \Theta(x, y) $ is convex in $ x $ for any fixed $ y $, and so is the value function $ \varphi(x) = \max_{y \in \{0,1\}^n} \Theta(x, y) $.

Based on Step 3 of Algorithm~\ref{alg:RLLS_submax} and the convexity of the set $ \mathcal{X} $, we obtain for any $ x \in \mathcal{X} $ and $ k = 0, \ldots, K-1 $:
\begin{align*}
\|x_{k+1} - x\|^2
&\le \|x_k - K^{-\frac{1}{2}} \nabla_x \Theta(x_k, y_k) - x\|^2 \\
&= \|x_k - x\|^2 + \| K^{-\frac{1}{2}} \nabla_x \Theta(x_k, y_k) \|^2 - 2\left\langle K^{-\frac{1}{2}} \nabla_x \Theta(x_k, y_k), x_k - x \right\rangle \\
&\le \|x_k - x\|^2 - 2\left\langle K^{-\frac{1}{2}} \nabla_x \Theta(x_k, y_k), x_k - x \right\rangle + \frac{L^2}{K}.
\end{align*}

From the convexity of $ \Theta(\cdot, y_k) $, for any $x \in \mathcal X$ there holds
\begin{align*}
\Theta(x_k, y_k) - \Theta(x, y_k)
&\le \left\langle \nabla_x \Theta(x_k, y_k), x_k - x \right\rangle \\
&\le \frac{K^{\frac{1}{2}}}{2} \left( \|x_k - x\|^2 - \|x_{k+1} - x\|^2 \right) + \frac{L^2}{2K^{\frac{1}{2}}}.
\end{align*}

Summing over $ k = 0 $ to $ K-1 $, we get
\begin{eqnarray}
\nonumber \sum_{k=0}^{K-1}  \Theta(x_k, y_k) - \Theta(x, y_k)
&\le & \frac{K^{\frac{1}{2}}}{2} \|x_0 - x\|^2 + \frac{L^2}{2K^{\frac{1}{2}}} \cdot K \\
&\le & \frac{K^{\frac{1}{2}}}{2} (D^2 + L^2). \label{eq:convex_sum}
\end{eqnarray}

Furthermore, since subsolver $\alg $ is a $ \gamma$-approximation algorithm, we have
$$
\Theta(x_k, y_k) \ge \gamma \varphi(x_k) , \quad \forall k = 0,\ldots,K-1.
$$

Summing this inequality over all $ k $, we obtain
\begin{equation}
    \sum_{k=0}^{K-1} \gamma \varphi(x_k) - \Theta(x_k, y_k)  \le 0. \label{eq:submodular_sum}
\end{equation}

Combining inequalities~\eqref{eq:convex_sum} and~\eqref{eq:submodular_sum}, we deduce
$$
\sum_{k=0}^{K-1}  \gamma \varphi(x_k) - \Theta(x, y_k)  \le \frac{K^{\frac{1}{2}}}{2} (D^2 + L^2) .
$$

Using the convexity of $ \varphi $ and the definition of $\hat x$, we further have
\begin{align*}
\gamma \varphi(\hat{x})
&\le \frac{1}{K} \sum_{k=0}^{K-1} \gamma \varphi(x_k) \\
&\le \frac{1}{K} \sum_{k=0}^{K-1} \Theta(x, y_k) + \frac{D^2 + L^2}{2K^{\frac{1}{2}}}  \\
&\le \varphi(x) + \frac{D^2 + L^2}{2K^{\frac{1}{2}}} .
\end{align*}
By setting \(K =
\left\lceil \left(\frac{D^2+L^2}{2\epsilon}\right)^2 \right\rceil\), we obtain
\[
\Theta(\hat{x},\hat y)
\le \varphi(\hat{x}) \le
\frac{1}{\gamma}\min_{x\in\mathcal X}\varphi(x)
+\frac{\epsilon}{\gamma}
=
\frac{1}{\gamma}\min_{x\in\mathcal X}\max_{y\in\mathcal Y}\Theta(x,y)
+\frac{\epsilon}{\gamma}.
\]
Finally, applying the subsolver ALG to maximize \(\Theta(\hat{x},\cdot)\) over \(\mathcal Y\), we obtain
\[
\Theta(\hat{x},\hat y)\ge \gamma\varphi(\hat{x})
=
\gamma\max_{y\in\mathcal Y}\Theta(\hat{x},y).
\]
Therefore, \((\hat{x},\hat y)\) is a \((\gamma,\epsilon/\gamma)\)-approximate minimax point.
\end{proof}

Lemma \ref{thm:alg_PG} requires that the subsolver $\alg$ satisfies approximation guarantee when solving the inner maximization problem individually. In this context, we consider invoking algorithms for solving supermodular maximization (submodular minimization) problems, which are known to be polynomial-time solvable.

\begin{theorem}[\cite{jiang2022minimizing}, Theorem 1.7] \label{thm:supermodular_alg}
Given an Evaluation Oracle (EO) for a submodular function $h: \{0,1\}^n \to \mathbb{R}$, there exists a strongly polynomial-time algorithm that minimizes $h$ using $O\left(n^3 \log \log n / \log n\right)$ calls to EO.
\end{theorem}

{The strongly polynomial submodular minimization algorithm of \cite{jiang2022minimizing} applies Algorithm 3 of that paper to the Lov\'{a}sz extension of the target submodular function. For clarity, we use this exact submodular minimization routine as an inner oracle and refer to it as the \emph{Lov\'{a}sz-extension-based SFM (Submodular Function Minimization) oracle}.}
Together with earlier results, this implies the polynomial-time convergence of Algorithm \ref{alg:RLLS_submax} in the supermodular linear case.

\begin{theorem}
\label{theorem:supermodular_case}
Assume that $F$ is affine and $C$ is acute. Set $ K = \left\lceil \left(\frac{D^2 + L^2}{2\epsilon}\right)^2 \right\rceil $ for a given $\epsilon > 0$.
For the \ref{eq:binary_RLS} problem, Algorithm \ref{alg:RLLS_submax} with $\alg$ being the {Lov\'{a}sz-extension-based SFM oracle} in \cite{jiang2022minimizing} outputs an $\epsilon$-global minimax point in polynomial time.
\end{theorem}

\begin{proof}
According to Theorem \ref{thm:supermodular_alg} and Propositions  \ref{prop:sub_sup}, \ref{prop:lovasz extension}, when $C$ is acute, the {Lov\'{a}sz-extension-based SFM oracle} serves as an exact (i.e., 1-approximation) maximization oracle for $\Theta(x,\cdot)$ on $\{0,1\}^n$ with any $x \in \mathcal X$. Therefore, by Lemma \ref{thm:alg_PG}, Algorithm \ref{alg:RLLS_submax} with this  subsolver returns an $\epsilon$-global minimax point after $K = \left\lceil \left(\frac{D^2 + L^2}{2\epsilon}\right)^2 \right\rceil$ iterations.
\end{proof}
	
{In this supermodular case, the Lov\'{a}sz-extension-based SFM oracle gives a clean polynomial-time exact oracle and a transparent worst-case complexity bound. For large-scale implementations, it can be replaced by some practical alternatives including the Fujishige-Wolfe minimum-norm-point method~\cite{fujishige2005submodular,wolfe1976finding,chakrabarty2014provable}, active-set methods for submodular minimization~\cite{kumar2017active}, and first-order or bundle-type methods applied to the Lov\'{a}sz extension~\cite{bach2013learning}. If these methods are terminated approximately, then the outer convergence analysis would need to account for the resulting inner error.}

Based on the theoretical results of Algorithm~\ref{alg:RLLS_submax}, we can formalize the advantage of our methodological framework when applied to solve the hypercube type uncertainty model~\ref{eq:binary_RLS_chen} compared to continuous minimax optimization methods.

\begin{remark}
For the supermodular linear case, we can conclude that the model \ref{eq:binary_RLS_chen} is a smooth convex-nonconcave minimax problem, allowing us to employ the state-of-the-art {alternating gradient projection (AGP)} algorithm in \cite{xu2023unified} for its solution. In terms of the convergence, Algorithm~\ref{alg:RLLS_submax} yields an $\epsilon$-global minimax point, whereas AGP can only achieve an $\epsilon$-stationary point. Regarding iteration complexity, the complexity of Algorithm~\ref{alg:RLLS_submax} is $O(\epsilon^{-2})$. This matches the complexity of AGP when $\Theta$ is strongly convex in $x$, which cannot be guaranteed in our problem. For the case where $\Theta$ is  convex in $x$, the iteration complexity of AGP is $O(\epsilon^{-4})$.
\end{remark}

		\subsection{Supermodular Nonlinear \ref{eq:binary_RLS}}
In this subsection we consider the case that $F$ is nonlinear and differentiable. The lack of guaranteed convexity of $\Theta$ with respect to $x$ leads to a modification of the outer algorithmic framework, as shown in Algorithm~\ref{alg:RNLS}. Unlike Algorithm \ref{alg:RLLS_submax} which returns the mean of trajectory points, this algorithm returns \( \hat{x} \) selected uniformly from trajectory points.

	\begin{algorithm}[htbp]
		\caption{Projected Gradient Algorithm for Nonlinear \ref{eq:binary_RLS}}
		\label{alg:RNLS}
		\begin{algorithmic}[1] 
			\REQUIRE $\Theta$, $x_0\in {\cal X}$, $K$, $\mu$
			\FOR{$k = 0, \ldots, K$}
			\STATE  $y_k \leftarrow \alg(\Theta(x_k,\cdot))$;
			\STATE $x_{k+1} =
					\operatorname{Proj}_{\mathcal X}(
			x_k - \mu \nabla_x \Theta(x_k,y_k)
					)
			$;
			\ENDFOR
			\STATE Randomly select \(\hat{x}\) from the set
\(\{x_k\}_{k = 0} ^{K}\)  with uniform distribution;
			\STATE $ \hat y\leftarrow \alg(\Theta(\hat x,\cdot))$;
			\RETURN $\hat x$,  $\hat y$.
		\end{algorithmic}
	\end{algorithm}

Since $\mathcal{X}$ is compact and convex and $F$ is differentiable, for every fixed $y \in \mathcal{Y}$, the function $\Theta(\cdot, y)$ and its gradient $\nabla_x \Theta(\cdot, y)$ are globally Lipschitz continuous over $\mathcal{X}$. We denote by $L$ and $\ell$ the uniform Lipschitz constants of $\Theta(\cdot, y)$ and $\nabla_x \Theta(\cdot, y)$, respectively, valid for all $y \in \mathcal{Y}$.
Under the above condition and notations, it is obvious that the value function of \ref{eq:binary_RLS} \( \varphi \) is \( \ell \)-weakly convex, meaning that \( \varphi(x) + \frac{\ell}{2} \|x\|^2 \) is convex over \( \mathcal{X} \). This enables us to define an \( \epsilon \)-stationary point for the problem via the Moreau envelope.
{For weakly convex and possibly nonsmooth objectives, the gradient norm of the Moreau envelope has become a standard stationarity measure, as it provides a smooth surrogate whose small gradient indicates near-stationarity of a nearby proximal point; see, e.g., \cite{davis2018stochastic,davis2019stochastic}.}

    \begin{definition}
    \label{def:moreau_envelope}
     The \emph{Moreau envelope} of $\varphi$ with a parameter $\lambda >0$ is defined as
    \[
    \varphi_{\lambda}(x) = \min_{w \in \mathcal X} \left\{ \varphi(w) + \frac{1}{2\lambda} \|w - x\|^2 \right\},
    \]
    with its associated proximity operator
       \[
   \operatorname{prox}_{\varphi,\lambda}(x) = \arg\min_{w \in \mathcal X} \left\{ \varphi(w) + \frac{1}{2\lambda} \|w - x\|^2 \right\}.
    \]
\end{definition}

Since $\varphi$ is $\ell$-weakly convex, we have $\operatorname{prox}_{\varphi, \frac{1}{2\ell}}$ is single-valued and thus $\varphi_{\frac{1}{2\ell}}$ is differentiable \cite[Proposition 13.37]{rockafellar1998variational}, allowing us to study $\epsilon$-stationary points of \ref{eq:binary_RLS} using the Moreau envelope.

Denote $\Delta = \varphi(x_0) - \min_{x \in \mathcal X} \varphi(x)$.
The following theorem gives the convergence and complexity of Algorithm \ref{alg:RNLS}.
    \begin{theorem}\label{thm:convergence_stationary}
    Assume that  $F$ is differentiable and $ C $ is acute. Set
    \[
    \mu = \left( \frac{\Delta}{L^2 \ell (K+1)} \right)^{\frac{1}{2}} {\rm ~~with ~~}    K =  \left \lfloor   \frac{64 L^2 \ell \Delta }{\epsilon^4}\right \rfloor,
    \]
    for a given $\epsilon >0$. Then Algorithm \ref{alg:RNLS} with $\alg$ being the {Lov\'{a}sz-extension-based SFM oracle} in \cite{jiang2022minimizing} outputs a point \((\hat{x}, \hat{y})\) with \( \hat{x} \) satisfying
\begin{equation}\label{stationary3}
    \mathbb{E} \left[ \|\nabla \varphi_{\frac{1}{2\ell}}(\hat{x})\| \right] \leq \epsilon.
    \end{equation}
   \end{theorem}

	\begin{proof}
		Denote $\hat x_{k} = \operatorname{prox}_{\varphi,\frac 1 {2\ell} } (x_{k})$. For $k = 1,\ldots,K+1$, we have
		\begin{eqnarray}
			\nonumber \| \hat x_{k-1} -  x_{k} \|^2 & = &\| \operatorname{Proj}_{\mathcal X} (\hat x_{k-1}) - \operatorname{Proj}_{\mathcal X}(x_{k-1} - \mu \nabla_x \Theta(x_{k-1}, y_{k-1}) )\|^2 \\
         \nonumber   & \le & \| \hat x_{k-1} - x_{k-1} + \mu \nabla_x \Theta(x_{k-1}, y_{k-1}) \|^2 \\
			\nonumber & \le & \|\hat x_{k-1} -  x_{k-1}\|^2 + 2 \mu \left \langle  \hat x_{k-1} - x_{k-1} , \nabla_x \Theta(x_{k-1}, y_{k-1})\right\rangle + \mu^2 L^2.
		\end{eqnarray}
		Together by the definition of $\varphi_{\frac 1 {2\ell}}$ we obtain
		\begin{eqnarray}
			\nonumber	\varphi_{\frac 1 {2\ell}} (x_k) & \le & \varphi(\hat x_{k-1}) + \ell \|\hat x_{k-1} - x_{k}\|^2 \\
			\nonumber & = & \varphi_{\frac 1 {2\ell}}  ( x_{k-1})  - \ell  \|\hat x_{k-1} - x_{k-1}\|^2 + \ell \|\hat x_{k-1} - x_{k}\|^2 \\
			& \le & \varphi_{\frac 1 {2\ell}}  ( x_{k-1}) + 2 \mu \ell \left \langle  \hat x_{k-1} - x_{k-1} , \nabla_x \Theta(x_{k-1}, y_{k-1})\right\rangle + \mu ^2 L^2 \ell . \label{eq:nonlinear_lem_1}
		\end{eqnarray}
		By the $\ell$-Lipschitz continuity of $\nabla_x\Theta(\cdot,y_{k-1})$ we have
		\begin{eqnarray}
			\nonumber &&\left\langle \hat x_{k-1} - x_{k-1} , \nabla_x \Theta(x_{k-1}, y_{k-1}) \right\rangle \\
          \nonumber  & \le & \Theta(\hat x_{k-1}, y_{k-1} ) -  \Theta( x_{k-1}, y_{k-1} )  + \frac \ell 2   \|\hat x_{k-1} - x_{k-1}\|^2 ,
		\end{eqnarray}
        Since $\alg$ is a $1$-approximation algorithm, then
		\begin{eqnarray}
			\nonumber \Theta(\hat x_{k-1}, y_{k-1} ) -  \Theta( x_{k-1}, y_{k-1}) & \le & \varphi(\hat x_{k-1}) - \varphi(x_{k-1}) \\
			\nonumber  & \le &  \varphi( x_{k-1}) - \ell \|\hat x_{k-1} - x_{k-1}\|^2 - \varphi(x_{k-1})  \\
			& = &  - \ell \|\hat x_{k-1} - x_{k-1}\|^2  . \label{eq:nonlinear_lem_2}
		\end{eqnarray}
		Thus together by (\ref{eq:nonlinear_lem_1}), (\ref{eq:nonlinear_lem_2}) and the fact $2 \ell\|\hat x_{k-1} - x_{k-1}\| = \|\nabla \varphi_{\frac 1 {2\ell}}(x_{k-1})\|$ we have
		\begin{eqnarray}
			\nonumber	&& \varphi_{\frac 1 {2\ell}} (x_k) - 	\varphi_{\frac 1 {2\ell}} (x_{k-1}) \\
\nonumber & \le &
			2 \mu \ell   \left[ \Theta(\hat x_{k-1}, y_{k-1} ) -  \Theta( x_{k-1}, y_{k-1}) \right]   + \mu \ell^2   \|\hat x_{k-1} - x_{k-1}\|^2  + \mu ^2 L^2 \ell \\
			\nonumber & \le &   - \mu \ell^2   \|\hat x_{k-1} - x_{k-1}\|^2  + \mu ^2 L^2 \ell  \\
			& = &  -  \frac{\mu}{4}  \|\nabla \varphi_{\frac 1 {2\ell}}(x_{k-1})\|^2  + \mu ^2 L^2 \ell. \label{eq: varphi_difference_bound}
		\end{eqnarray}

        Summing $(\ref{eq: varphi_difference_bound})$ from $k = 1$ to $K + 1$ yields
		\begin{eqnarray}
		\nonumber	\varphi_{\frac 1 {2 \ell}} (x_{K +1}) - 	\varphi_{\frac 1 {2 \ell}} (x_{0})  & \le &  -  \frac{\mu}{4}  \sum_{k = 0} ^ {K}\|\nabla \varphi_{\frac 1 {2\ell}}(x_{k})\|^2  + \mu ^2 L^2 \ell(K+1).
		\end{eqnarray}
		Thus
		\begin{eqnarray}
		\nonumber	\frac 1 {K +1 }\sum_{k = 0} ^ {K}\|\nabla \varphi_{\frac 1 {2\ell}}(x_{k})\|^2 & \le & \frac{4 \Delta}{\mu (K+1)} + 4 \mu L^2 \ell.
		\end{eqnarray}
		Substituting $\mu = \left( \frac{\Delta}{L^2\ell (K+1)}  \right)^{\frac12}$ and $K = \left \lfloor   \frac{64 L^2 \ell \Delta }{\epsilon^4}\right \rfloor$ yields
		\begin{eqnarray}
\nonumber \mathbb{E}\big[\|\nabla \varphi_{\frac{1}{2\ell}}(\hat{x})\|\big] & \le & \left(\mathbb{E}\big[\|\nabla \varphi_{\frac{1}{2\ell}}(\hat{x})\|^2\big] \right)^{\frac 12} \\
    	\nonumber	& = &	\left( \frac 1 {K +1 }\sum_{k = 0} ^ {K}\|\nabla \varphi_{\frac 1 {2\ell}}(x_{k})\|^2 \right)^{\frac 12} \le \epsilon.
		\end{eqnarray}
 We complete the proof.
	\end{proof}

The inequality (\ref{stationary3}) means that a randomly selected point $\hat{x}$ with the uniform distribution from the iterates $\{x_k\}_{k=0}^{K}$ generated by Algorithm \ref{alg:RNLS} is an $\epsilon$-stationary point in expectation, as measured by the gradient norm of the Moreau envelope \cite{davis2018stochastic}.

{\section{Non-Supermodular BRLS}\label{sec:non-supermodular}

This section considers BRLS beyond the supermodular setting. We begin with the linear submodular case, where the inner maximization is an unconstrained submodular maximization problem and can be handled by a double-greedy approximation oracle. Moreover, we discuss this approach for nonlinear submodular case by linearization of the nonlinear function $F$ at iterates. Finally, we  treat the linear case with a general noise propagation matrix through an SDP-based approximation oracle. }

\subsection{Submodular Linear \ref{eq:binary_RLS}} \label{subsec:submodular linear}

In this subsection, we explore another scenario in which $C$ is obtuse, meaning that $\Theta(x,\cdot)$ is submodular for any $x \in \mathcal X$ according to Proposition~\ref{prop:sub_sup}.
We focus on the case that $F$ is affine.
For this case, we can still employ Algorithm~\ref{alg:RLLS_submax} as the outer-loop framework.
Unlike the supermodular case, computing an exact solution in $\operatorname{SOL}(\Theta(x,\cdot), \mathcal{Y})$ for a given $x$ is unknown to be solvable in polynomial time in this case. Nevertheless, efficient approximation algorithms for submodular maximization allow us to compute approximate solutions with provable guarantees.
We adopt the deterministic double greedy method from \cite{buchbinder2015tight} described in Algorithm~\ref{alg:DDG}, and provide full theoretical guarantees tailored to our setting.

\begin{algorithm}[h]
\caption{Double Greedy Algorithm ($\operatorname{DG}(\Theta(x,\cdot))$)}
\label{alg:DDG}
\begin{algorithmic}[1]
\REQUIRE  $ \Theta(x, \cdot) $.
\STATE $ \underline{y}^0 \leftarrow \mathbf{0} $, $ \overline{y}^0 \leftarrow \mathbf{1} $;
\FOR{$ k = 1, \ldots, n $}
    \STATE $ a \leftarrow \Theta(x, \underline{y}^{k-1} + \mathbf{e}_k) - \Theta(x, \underline{y}^{k-1}) $;
    \STATE $ b \leftarrow \Theta(x, \overline{y}^{k-1} - \mathbf{e}_k) - \Theta(x, \overline{y}^{k-1}) $;
    \IF{$ a \geq b $}
        \STATE $ \underline{y}^k \leftarrow \underline{y}^{k-1} + \mathbf{e}_k $;
        \STATE $ \overline{y}^k \leftarrow \overline{y}^{k-1} $;
    \ELSE
        \STATE $ \underline{y}^k \leftarrow \underline{y}^{k-1} $;
        \STATE $ \overline{y}^k \leftarrow \overline{y}^{k-1} - \mathbf{e}_k $;
    \ENDIF
\ENDFOR
\RETURN $ \hat{y} = \underline{y}^n = \overline{y}^n $.
\end{algorithmic}
\end{algorithm}
	
Algorithm \ref{alg:DDG} generates two sets of iterates: one starting from $ \underline{y}^0 = \mathbf{0} $ and the other from $ \overline{y}^0 = \mathbf{1} $. At each step, it compares marginal gains to determine whether to set the current component to 1 or 0. Algorithm \ref{alg:DDG} terminates after $n$ iterations at a feasible solution $\hat{y}=\overline{y}^n=
\underline{y}^n$.
	The following lemma provides the approximation guarantee for this algorithm.
	
	\begin{lemma}\label{lem:DG}
When \( C \) is obtuse, Algorithm~\ref{alg:DDG} is a \( \frac{1}{3} \)-approximation algorithm for maximizing \( \Theta(x, \cdot) \) over \( \{0,1\}^n \) for any fixed \( x \in \mathcal{X} \). Moreover, when \( C \) is column-orthogonal, Algorithm~\ref{alg:DDG} returns an exact optimal solution to the maximization problem.
	\end{lemma}
	\begin{proof}
When \( C \) is obtuse, Proposition~\ref{prop:sub_sup} implies that \( \Theta(x, \cdot) \) is submodular for any \( x \in \mathcal{X} \). This recovers the setting of \cite{buchbinder2015tight}, and for readability, we present a streamlined adaptation of the analysis. By the fact
		\begin{eqnarray}
			\nonumber (\underline{y}^{k-1} + \mathbf{e}_k) \lor (\overline{y}^{k-1} - \mathbf{e}_k) &=& \overline{y}^{k-1}, \\
			\nonumber  (\underline{y}^{k-1} + \mathbf{e}_k) \land (\overline{y}^{k-1} - \mathbf{e}_k) &=& \underline{y}^{k-1},
		\end{eqnarray}
		we have
        	\begin{equation*}
    	a +b  =  \Theta(x, \underline{y}^{k-1} + \mathbf{e}_k) + \Theta(x, \overline{y}^{k-1} - \mathbf{e}_k ) - \Theta(x, \underline{y}^{k-1}  ) - \Theta(x, \overline{y}^{k-1} ) \ge 0.    	
        	\end{equation*}
        				For an optimal solution $y^* \in \operatorname{SOL}(\Theta(x,\cdot),\mathcal Y)$, let $y^k := (y^* \lor \underline{y}^k) \land \overline{y}^k$ for $k = 1,\ldots,n$. Then we yield $y^0 = y^*$ and $y^n = \underline{y}^n = \overline{y}^n = \hat y$.
		Without loss of generality, assume that $a \ge b$ at the $k$-th iteration and hence $y^k = y^{k-1} + \mathbf{e}_k$, $\underline{y}^k = \underline{y}^{k-1} + \mathbf{e}_k$ and $\overline{y}^k = \overline{y}^{k-1}$.
		By the submodularity of $\Theta(x,\cdot)$ and the fact $y^{k-1} \ge \underline{y}^{k-1}$ we have
		\begin{equation}\label{eq:lem2_proof_1}
			\Theta(x,y^{k-1}) - \Theta(x,y^{k}) \le \Theta(x,\underline{y}^k) - \Theta(x,\underline{y}^{k-1}).
		\end{equation}
		By summing equation (\ref{eq:lem2_proof_1}) from \( k = 1 \) to \( n \), we obtain
		\[
		\Theta(x,y^*) - \Theta(x,y^n) \le \Theta(x,y^n) - \Theta(x,\mathbf{0}) + \Theta(x,y^n) - \Theta(x,\mathbf{1}).
		\]
By the fact $\Theta(x,y) \ge 0$ for any $x \in \mathcal X, y \in \mathcal Y$ and $y^n = \hat y$ we obtain that $\hat y$ is a $\frac 13$-approximate solution.

When \( C \) is column-orthogonal, Proposition~\ref{prop:sub_sup} implies that \( \Theta(x, \cdot) \) is modular for any \( x \in \mathcal{X} \), which is equivalent to its separability in \( y \). That is, maximizing \( \Theta(x, \cdot) \) over \( y \in \{0,1\}^n \) reduces to independently selecting, for each coordinate \( y_k \), the value in \( \{0,1\} \) that maximizes the function.
Moreover, due to modularity, at each step \( k \) of Algorithm \ref{alg:DDG}, we have
\[
a = -b = \Theta(x, \mathbf{e}_k) - \Theta(x, \mathbf{0}) = \frac 12 \|c_k\|^2 - (C^\top F(x))_k.
\]
The algorithm's decision rule thus corresponds exactly to choosing the optimal value for each coordinate independently. Consequently, Algorithm \ref{alg:DDG} returns an optimal solution to the maximization problem.
	\end{proof}

Based on Lemma \ref{lem:DG}, when \( C \) is column-orthogonal, Algorithm \ref{alg:DDG}'s coordinate-wise decisions exactly coincide with the strategy used in \cite{chen2024min} to obtain the closed-form solution of the inner maximization problem. Lemma \ref{lem:DG} thus extends their approach to the broader case where \( C \) is obtuse, and provides a provable approximation guarantee.

Combining Lemma \ref{thm:alg_PG} and Lemma \ref{lem:DG}, we directly establish the convergence and complexities of Algorithm~\ref{alg:RLLS_submax} equipped with the DG algorithm as the subsolver  when used for solving the submodular linear \ref{eq:binary_RLS} problem.
\begin{theorem} \label{thm:submodular_linear}
		Assume that $F$ is affine. Set $K = \left\lceil 9\left(\frac{D^2+L^2}{2\epsilon}\right)^2 \right\rceil
		$ for a given $\epsilon >0$. When $C$ is obtuse, Algorithm \ref{alg:RLLS_submax} with $\alg$ replaced by $\operatorname{DG}$ outputs a $(\frac 13 , \epsilon)$-approximate minimax point. When $C$ is column-orthogonal, Algorithm \ref{alg:RLLS_submax} with $\alg$ replaced by $\operatorname{DG}$ outputs an $\epsilon$-global minimax point.
	\end{theorem}

{
\begin{remark}[On the Submodular Nonlinear Case]
\label{rem:nonlinear_obtuse}
Theorem~\ref{thm:submodular_linear} shows that, when \(F\) is affine and \(C\) is obtuse,  Algorithm \ref{alg:RLLS_submax} with $\alg$  subsolver yields a \((1/3,\epsilon)\)-approximate minimax guarantee. If the affine assumption on \(F\) is removed, the inner problem remains submodular in \(y\) for every fixed \(x\), and hence the DG  algorithm can still be used as an inner subsolver. One practical route is to combine this subsolver with a local nonlinear least-squares method, such as a Gauss--Newton-type scheme, by linearizing \(F\) around the current iterate and applying the linear BRLS framework to the resulting local model. However, the proof of Theorem~\ref{thm:submodular_linear} relies on the convexity in \(x\) induced by affine \(F\), and therefore does not yield an approximate minimax guarantee for the nonlinear case.

Another possible route is to work with the relaxation form \ref{eq:binary_RLS2_chen}, viewed as a smooth nonconvex-nonconcave minimax problem. Existing methods for nonconvex-nonconcave minimax optimization may then be used to seek local equilibrium or stationarity guarantees, see, e.g., \cite{daskalakis2023stay,daskalakis2018limit,jin2020local}.
\end{remark}
}

{
\subsection{Linear BRLS with General Noise Correlation Matrices}
\label{sec:general-C}

In this subsection, we consider  a general matrix \(C\) in problem (\ref{eq:binary_RLS}), including mixed-angle matrices for which the off-diagonal entries of \(C^\top C\) have mixed signs. For a fixed \(x\), the inner maximization problem is a positive-semidefinite Boolean quadratic programming (BQP) problem. To illustrate that the classical SDP approximation result can apply to this kind of problem, we first convert it exactly into a homogeneous quadratic maximization problem over \(\{\pm1\}\) variables with an anchoring constraint. Let
\[
s=2y-\mathbf 1\in\{\pm1\}^{n} \quad {\rm and} \quad
v=\begin{pmatrix}s\\1\end{pmatrix}.
\]
Since \(y=(s+\mathbf 1)/2\), we have
\[
F(x)-Cy
 =
\left(F(x)-\frac12 C\mathbf 1\right)-\frac12 Cs
 =
Qv,
\qquad
Q:=\left[-\frac12 C,\; F(x)-\frac12 C\mathbf 1\right].
\]
Therefore, with \(H:=Q^\top Q\succeq0\), the inner maximization problem in (\ref{eq:binary_RLS}) is equivalently written as
\begin{equation}\label{eq:BQP}
	\max_{\substack{v \in \{\pm1\}^{n+1}\\ v_{n+1}=1}}
	\frac12 v^\top H v.
\end{equation}

The following result is a standard consequence of the SDP relaxation and hyperplane-rounding guarantee for positive-semidefinite Boolean quadratic maximization due to \cite{Nesterov1998Semidefinite}, together with the derandomization framework for SDP-based approximation algorithms in \cite{Mahajan1995Derandomizing}. We state it in the anchored form needed for \eqref{eq:BQP}. The anchoring constraint \(v_{n+1}=1\) causes no difficulty, since the objective is homogeneous quadratic and a simultaneous sign flip of all components preserves the objective value. The slack parameter \(\eta\) accounts for solving the SDP relaxation only to finite relative accuracy.

\begin{lemma}
	\label{lem:sdp-approx-bqp}
	For any fixed \(\eta\in(0,2/\pi)\), problem \eqref{eq:BQP} admits a deterministic polynomial-time algorithm, denoted by \(\mathrm{ALG}_{\mathrm{SDP}}^\eta\), that returns a vector \(\hat v\in\{\pm1\}^{n+1}\) with \(\hat v_{n+1}=1\) and
	\[
	\frac12 \hat v^\top H\hat v
	\ge
	\left(\frac{2}{\pi}-\eta\right)
	\max_{\substack{v\in\{\pm1\}^{n+1}\\ v_{n+1}=1}}
	\frac12 v^\top Hv .
	\]
	Consequently, after mapping \(\hat y=(\hat v_{1:n}+\mathbf 1)/2\), the same algorithm provides a \((2/\pi-\eta)\)-approximation oracle for maximizing \(\Theta(x,\cdot)\) over \(\{0,1\}^n\), for any fixed \(x\in\mathcal X\).
\end{lemma}

Indeed, let \(Z^*\) denote the optimal value of the SDP relaxation.
For any prescribed \(\delta\in(0,1)\), standard interior-point methods can compute a feasible SDP solution with objective value at least \((1-\delta) Z^*\) in time polynomial in the problem size and \(\log(\delta^{-1})\); see Theorem 5.1 of \cite{vandenberghe1996semidefinite}.
Applying the Nesterov rounding guarantee to such an approximate SDP solution gives the ratio \((2/\pi)(1-\delta)\). Choosing \(\delta\) such that \((2/\pi)\delta\le \eta\) yields the stated \((2/\pi-\eta)\) guarantee.

In practical implementations, one may also solve the inner BQP using off-the-shelf global optimization solvers such as Gurobi. Such solvers can be useful for moderate-scale instances, but their worst-case complexity is exponential in general; hence they are not used in our theoretical oracle guarantees.

	With a constant-factor approximation algorithm for the general inner problem in hand, we can now apply Algorithm~\ref{alg:RLLS_submax} to the linear BRLS setting.
	
	\begin{corollary}
		\label{thm:general-bqls-minimax}
Assume that ${F}$ is affine. For any fixed $\eta\in(0,2/\pi)$, by employing $\alg = \mathrm{ALG}_{\mathrm{SDP}}^\eta$ (as described in Lemma~\ref{lem:sdp-approx-bqp}) in Algorithm~\ref{alg:RLLS_submax}, one can compute a $(\frac 2 \pi-\eta, \frac{\epsilon}{2/\pi-\eta})$-approximate minimax point for problem (\ref{eq:binary_RLS}). The iteration complexity remains $O(\epsilon^{-2})$ for any fixed $\eta$, while the oracle cost depends polynomially on the SDP accuracy required to realize this $\eta$.
	\end{corollary}

\begin{remark}[Models with Operator Noise]
\label{rem:operator_noise}
The above results for problem (\ref{eq:binary_RLS}) can  be extended to models with uncertainty in the operator. For example, one may consider
\[
    \min_{x \in \mathcal X}
    \max_{y \in \mathcal Y,\, z \in \mathcal Z}
    \left\|
    \left(A + \sum_{i=1}^{n_2} z_i A_i\right)x - b - C y
    \right\|^2 ,
\]
where \(\mathcal Y\) is either \(\{0,1\}^{n_1}\) or \([0,1]^{n_1}\), and \(\mathcal Z\) is either \(\{0,1\}^{n_2}\) or \([0,1]^{n_2}\). For each fixed \(x\), the inner maximization can still be reduced to a Boolean quadratic maximization. Therefore, the SDP-based oracle can still be applied to obtain a constant-factor approximation for the inner problem. However, the perturbation directions now include the \(x\)-dependent vectors \(A_i x\), so the fixed acute/obtuse structure that drives the submodular and supermodular oracles in this paper is generally absent.
This illustrates the role of the SDP oracle as a general-purpose fallback, while the sharper structural results of this paper rely on the sign structure of \(C^\top C\).
\end{remark}

\begin{remark}[Comparison of Inner Solvers]
\label{rem:algo_comparison}
The three inner solvers play different theoretical and computational roles. The Lov\'asz-extension-based SFM oracle and the double-greedy oracle exploit the sign structure of \(C^\top C\): the former gives exact inner maximization in the acute, supermodular case, while the latter gives a lightweight \(1/3\)-approximation in the obtuse, submodular case.
The SDP-based oracle is reserved for unstructured \(C\), where it provides a general constant-factor guarantee at a higher computational cost.
When the acute/obtuse geometry is available, the structural oracles are preferable because they use well-developed submodular-optimization tools.
\end{remark}
}

\section{Numerical Experiments}\label{sec:numerical}

{This section presents numerical experiments that illustrate the efficiency of the proposed BRLS model for handling representative structured uncertainty. The experiments cover a real label-corruption task, a synthetic linear BRLS task, and a nonlinear thresholded phase-retrieval task with missing binary labels. They are intended to complement the theoretical oracle-based guarantees by showing the robustness of the computed solutions of the BRLS model in representative settings, rather than to provide a comprehensive benchmark against the full range of task-specific noisy learning.}
The MATLAB codes of this section are publicly available in the GitHub repository:
\url{https://github.com/zhyg1212/Robust-Least-Squares-Problems-with-Binary-Uncertain-Data}.

\subsection{Health Status Prediction Using Wearable Sensor Data}
\label{sec:health_prediction}

 In this experiment, we evaluate the robustness of the \ref{eq:binary_RLS} model on a binary health status classification task using wearable sensor data under adversarial label corruption. The dataset is derived from \cite{chen2025differential}, comprising physiological and motion signals collected from  smart watches (e.g., heart rate, Sp$\text{O}_2$) and smart insoles (e.g., foot pressure, step frequency, center of pressure) at 5-second intervals over 10 days from 10 elderly users. Each user contributes approximately \( r_i \approx 130{,}000 \) data points (\(i = 1, \dots, 10\)), and each data point is labeled as either 0 (healthy) or 1 (ill). {In the test sets, the positive-label ratio varies across users from 21.41\% to 58.87\%, with an average of 33.71\%. To avoid relying only on a single aggregate metric, we report average accuracy together with confusion matrices in a representative high-noise setting.}

To construct the training data, we aggregate features within homogeneous label groups: for each user, we partition $ r_i $ samples into $ r=10{,}000 $ clusters of equal size.
Then each row $a_j,~ j= 1,\ldots r$, in the training matrix $ A \in \mathbb{R}^{r \times 28} $ is computed as the average of the corresponding feature vectors within its cluster. The corresponding label $(b_{\rm true})_j$ is set to be the label of its corresponding cluster.
{We consider label-flipping ratios $\rho\in\{0,0.05,\ldots,0.40\}.$
For each \(\rho\), \(\lceil \rho r\rceil\) entries of \(b_{\rm true}\) are flipped to generate the observed corrupted training label vector \(b\in \mathbb{R}^{r}\).}

{The methods are provided with a candidate unreliable-label set \(\mathcal I\), which may imperfectly indicate where the label corruptions occur. The set \(\mathcal I\) contains two types of samples: flipped labels that are correctly marked as unreliable, and unflipped labels that are mistakenly marked as unreliable. Note that there are also flipped labels that are not included in \(\mathcal I\). We control the quality of \(\mathcal I\) by two experimental parameters. The coverage of flipped labels is set to either \(80\%\) or \(50\%\), meaning that \(\mathcal I\) contains either \(80\%\) or \(50\%\) of the flipped labels. The precision of \(\mathcal I\) is fixed at \(80\%\), meaning that \(80\%\) of the samples in \(\mathcal I\) are truly flipped and the remaining \(20\%\) are unflipped false positives. These coverage and precision values are used only to generate the experimental scenarios; the algorithms receive only the candidate set \(\mathcal I\), not the true flipped-label set or the true coverage and precision values.

We compare ordinary LS, discard-set LS, discard-set tuned LASSO, trimmed LS, and BRLS. Ordinary LS is trained on all corrupted labels.
These baselines are chosen to compare different least-squares-based ways of using, discarding, or robustifying the corrupted training labels, in line with the robust least-squares focus of this paper.
Discard-set LS uses the candidate corrupted-label set by solving LS after removing samples in ${\mathcal I}$. The discard-set tuned LASSO baseline solves
\[
	x_{\rm lasso}(\lambda)\in\arg\min_x \|Ax-b\|^2+\lambda\|x\|_1,
\]
after removing the same candidate set, where $\lambda$ is selected from a logarithmic grid using a validation split.
Trimmed LS is a robust regression baseline \cite{rousseeuw1984least} that keeps the smallest residuals:
\[
	x_{\rm trim}\in
	\arg\min_x \min_{\substack{I\subseteq [r]\\ |I|=h}}
	\sum_{i\in I}(a_i^\top x-b_{i})^2,
	\qquad h=\max\{m+1,\lfloor(1-\rho)r\rfloor\}.
\]
In our implementation, this problem is solved by iteratively refitting LS on the $h$ smallest-residual samples. Unlike the discard-set baselines and BRLS, trimmed LS does not use ${\mathcal I}$. The BRLS estimator is formulated according to \eqref{eq:binary_classification}. In this label-flipping setting, $C$ is diagonal and its nonzero diagonal entries correspond to samples in ${\mathcal I}$. Since $C$ is
acute and obtuse, the inner maximization has a closed-form coordinate-wise solution. Thus BRLS uses partial side information about which samples may be unreliable; it does not observe the corrected labels and does not flip corrupted labels back directly.
}

For testing, we use the original filtered features and their ground-truth labels, denoted by
\[
    A_{\rm test}\in\mathbb R^{r_{\rm test}\times 28},
    \qquad
    b_{\rm test}\in\{0,1\}^{r_{\rm test}},
\]
with $r_{\rm test} = \sum_{i=1}^{10}r_i$.
For a learned parameter \(x\in\mathbb R^{28}\), the predicted test labels are defined by
\[
(b_{\rm pred}(\rho,x))_j =
\begin{cases}
1, & \text{if } (A_{\rm test}x)_j>0.5,\\
0, & \text{otherwise},
\end{cases}
\qquad j=1,\ldots,r_{\rm test}.
\]
We report the test accuracy
\[
    \operatorname{Accuracy}(\rho,x)
    =
    1-\frac{1}{r_{\rm test}}
    \|b_{\rm pred}(\rho,x)-b_{\rm test}\|_1 .
\]
%

\begin{figure}[htbp]
	\centering
	\includegraphics[width=0.85\textwidth]{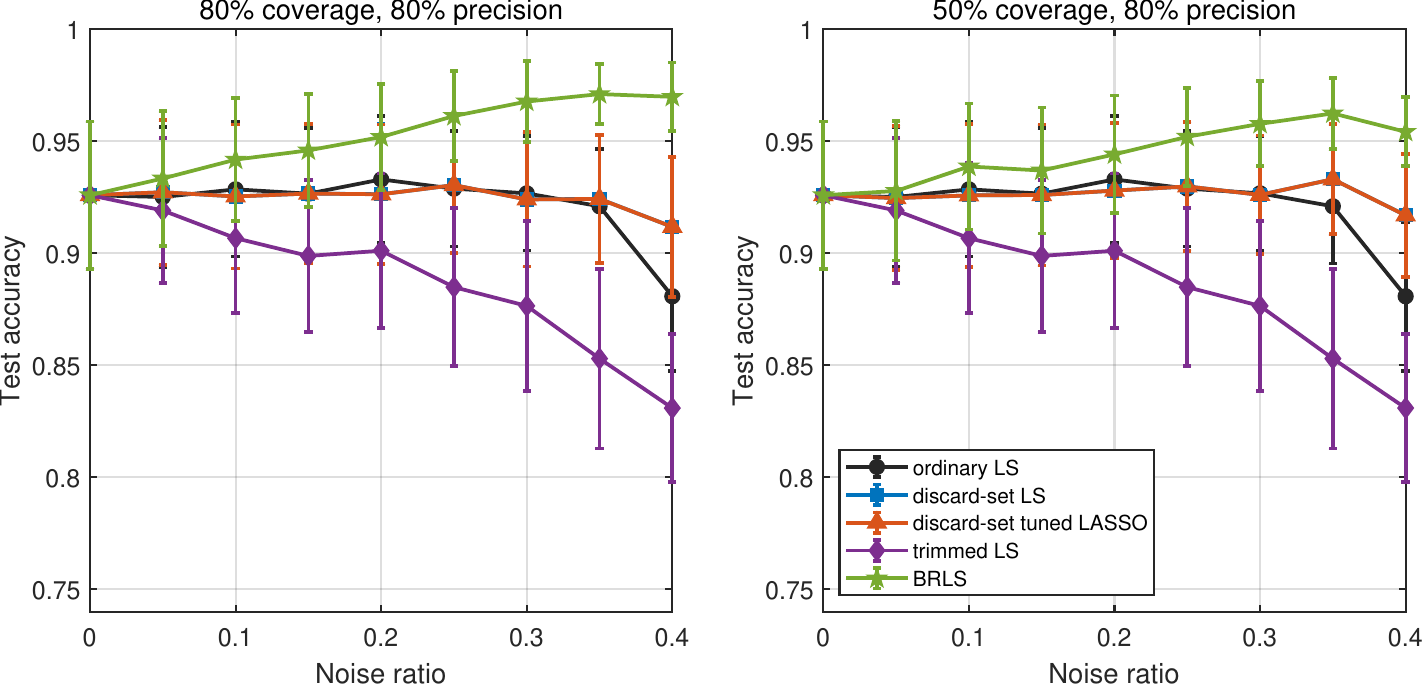}
	\caption{{Average accuracy over 10 users as a function of the label-flipping ratio with $r=10{,}000$. The candidate corrupted-label set has $80\%$ precision in both panels, and covers $80\%$ of the truly flipped labels in the left panel and $50\%$ in the right panel. The comparison includes ordinary LS, discard-set LS, discard-set tuned LASSO, trimmed LS, and BRLS.}}
	\label{fig:accuracy_comparison}
\end{figure}

\begin{figure}[htbp]
	\centering
	\includegraphics[width=1\textwidth]{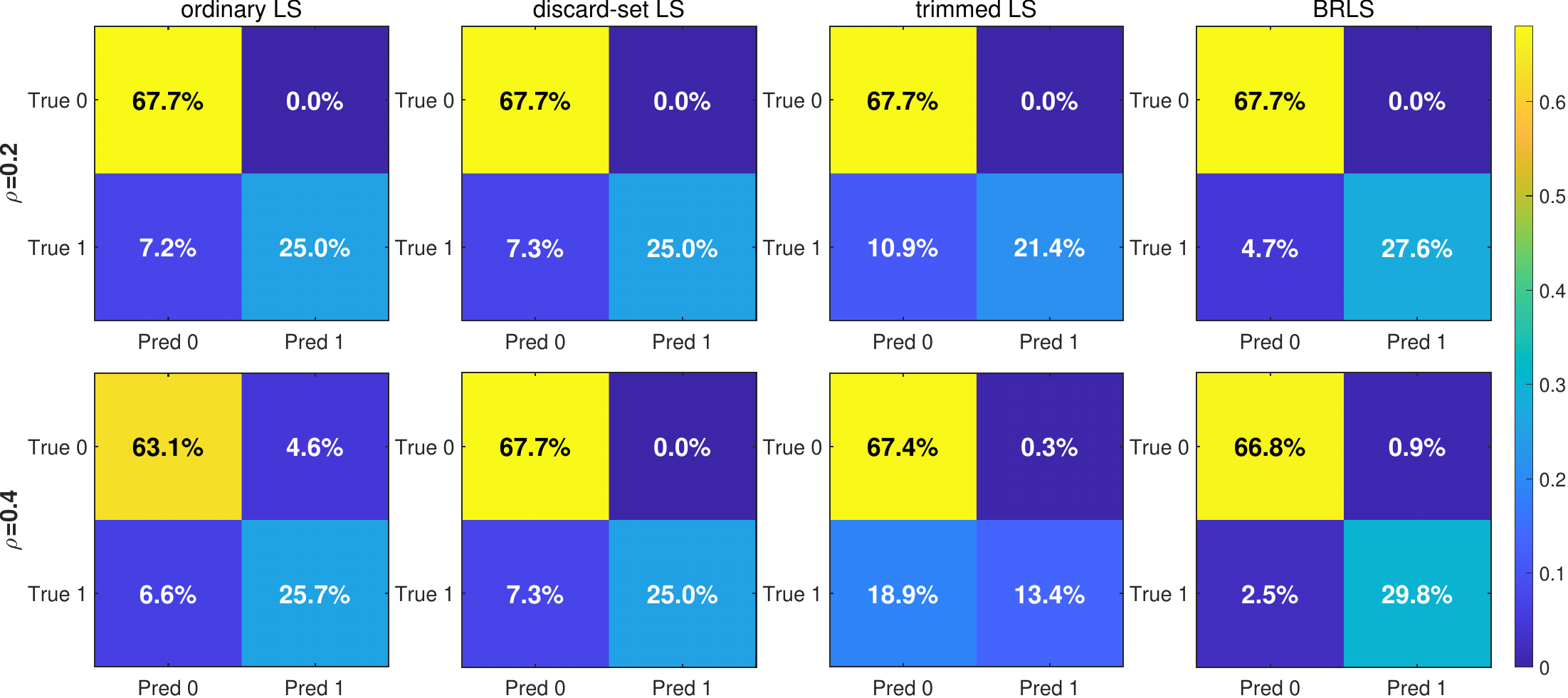}
	\caption{{Confusion matrices for the health-status experiment with $r=10{,}000$, $80\%$ coverage, and $80\%$ precision. The two rows correspond to $\rho=0.2$ and $\rho=0.4$, and all panels use the same color scale.}}
	\label{fig:heatmap_matrix}
\end{figure}

{
Figure~\ref{fig:accuracy_comparison} shows that all methods perform similarly when the label-flipping ratio is small. As $\rho$ increases, ordinary LS, discard-set LS, discard-set tuned LASSO, and trimmed LS degrade more rapidly, while BRLS retains higher average accuracy in these settings. This indicates that the BRLS formulation uses the candidate corrupted-label information differently from simply discarding suspected samples. Figure~\ref{fig:heatmap_matrix} further compares representative moderate- and high-noise cases. In particular, at $\rho=0.4$, BRLS substantially reduces false negatives and retains clearer diagonal dominance than the baselines.
}

{
\subsection{Synthetic Linear BRLS}\label{sec:controlled_hrls}

In this experiment, we consider a synthetic linear BRLS problem
\[
	\min_x \max_{y\in\{0,1\}^n} \frac12\|Ax-b-Cy\|^2,
	\qquad b=A x_{\rm true}+\xi ,
\]
where $x_{\rm true}\in\mathbb R^m$ is randomly generated and $\xi$ is independent Gaussian noise.  To avoid a diagonal or identity measurement model, we construct a high-dimensional measurement matrix $A\in\mathbb R^{r\times m}$ by embedding the structured noise directions $C\in\mathbb R^{r\times n}$ as part of the measurement directions and completing the remaining columns with sparse random normalized columns. The matrix $C$ is generated in three forms: acute, obtuse, and mixed-angle, by constructing normalized columns whose pairwise inner products are nonnegative, nonpositive, or mixed in sign, respectively. This design keeps the prescribed noise-correlation structure while testing the method on a non-identity linear map. We set $r=2m$ and use the five paired settings
\[
    (m,n)\in\{(100,30),(500,60),(1500,100),(3000,150),(5000,200)\}.
\]
For BRLS, the inner maximization is solved by the double-greedy subsolver. DG has a theoretical approximation guarantee in the submodular inner case; for the acute and mixed-angle cases in this experiment, we use it as a lightweight heuristic inner solver. Thus this experiment is intended as a scalability and structured-robustness test rather than a certified worst-case benchmark. We compare BRLS with LS.

Let
\(b_{\rm true}=Ax_{\rm true}\) denote the noise-free target signal. For a given solution \(x\), let
\(\hat y_{\rm DG}(x)\) be the output returned by the DG inner solver applied to
\[
    \max_{y\in\{0,1\}^n}\frac 12\|Ax-b_{\rm true}-Cy\|^2 .
\]
We define the DG-estimated adversarial residual error by
\[
    \operatorname{Err}(x)
    :=
    \frac{1}{2}\|Ax-b_{\rm true}-C\hat y_{\rm DG}(x)\|^2 .
\]
This quantity measures the squared residual after applying the DG-estimated worst-case binary perturbation.

Table~\ref{tab:controlled_hrls} reports the relative reduction of this adversarial residual error achieved by BRLS compared with LS:
\[
    \frac{
    \operatorname{Err}(x_{\rm LS})
    -
    \operatorname{Err}(x_{\rm BRLS})
    }{
    \operatorname{Err}(x_{\rm LS})
    }
    \times 100\%.
\]
BRLS consistently reduces the DG-estimated adversarial residual error for acute, obtuse, and mixed-angle matrices across all tested dimensions.
}

\begin{table}[htbp]
	\centering
	{
	\caption{DG-estimated adversarial residual error reduction of BRLS relative to LS in the synthetic high-dimensional BRLS experiment.}
	\label{tab:controlled_hrls}
	\begin{tabular}{ccccc}
		\hline
		$m$ & $n$ & acute (\%) & obtuse (\%) & mixed-angle (\%) \\
		\hline
		100  & 30  & 73.8 & 66.6 & 62.6 \\
		500  & 60  & 74.4 & 70.9 & 66.8 \\
		1500 & 100 & 73.5 & 71.9 & 68.1 \\
		3000 & 150 & 73.4 & 72.5 & 68.4 \\
		5000 & 200 & 73.3 & 72.6 & 69.7 \\
		\hline
	\end{tabular}
	}
\end{table}

\subsection{Thresholded Phase Retrieval with Missing Binary Labels}
\label{sec:robustness_comparison}

{
We finally consider a nonlinear missing-label problem motivated by thresholded phase retrieval. In standard phase retrieval, the measurements depend on the squared magnitudes $(a_i^\top x_{\rm true})^2$ rather than on the signed linear responses $a_i^\top x_{\rm true}$. In settings where only coarse or binary information is retained, a natural observation model is to record whether the intensity exceeds a prescribed threshold. We consider the binary labels
\[
(b_{\rm true})_i=\mathbf 1{\{(a_i^\top x_{\rm true})^2\geq 1/2\}},\quad i=1,\ldots,r.
\]
Labels close to the threshold are intrinsically less reliable, since small perturbations in the sensing vector, the signal, or the measurement process may change the binary decision. This makes thresholded phase retrieval a natural nonlinear testbed for binary labels that are partially missing or unknown.
Let $\mathcal I\subseteq [r]$ denote the samples whose labels are missing or unknown. For the remaining samples $i\notin\mathcal I$, an observed label $\widetilde b_i$ is available; a small portion of these known labels may still be erroneous.

We compare two least-squares-based ways of using the missing-label set. The first baseline is discard-set LS, which removes the samples in $\mathcal I$ and solves
\[
\min_x \sum_{i\notin\mathcal I}\left((a_i^\top x)^2-\widetilde b_i\right)^2.
\]

The second method is BRLS, which keeps the missing-label samples and lets the inner binary adversary choose the corresponding binary labels.
Note that the labels for \(i\in\mathcal I\) are unknown, so they cannot be directly used to form the observed-label vector in BRLS. To address this issue, we introduce an auxiliary nominal label vector \(\bar b\in\mathbb R^r\) and  define the diagonal matrix \(C\in\mathbb R^{r\times r}\) by
\[
\bar b_i=
\begin{cases}
\widetilde b_i, & i\notin\mathcal I,\\
0, & i\in\mathcal I,
\end{cases}
\qquad
c_{ii}=
\begin{cases}
0, & i\notin\mathcal I,\\
1, & i\in\mathcal I.
\end{cases}
\]
Then \(\bar b_i+c_{ii}y_i=\widetilde b_i\) for \(i\notin\mathcal I\), while
\(\bar b_i+c_{ii}y_i=y_i\in\{0,1\}\) for \(i\in\mathcal I\). Thus the unknown binary labels are covered by the same binary uncertainty mechanism as in \eqref{eq:binary_classification}. By setting
\[
(F(x))_i=(a_i^\top x)^2-\bar b_i,\qquad i=1,\ldots,r,
\]
this problem takes the form of problem (\ref{eq:binary_RLS})
\begin{equation}\label{sec5.3-1}
\min_x\max_{y\in\{0,1\}^r}\frac12\|F(x)-Cy\|^2 =
\min_x\max_{y\in\{0,1\}^r}\frac12\sum_{i=1}^{r}\left( (a_i^\top x)^2-\bar b_i - c_{ii}y_i \right)^2.
\end{equation}
Since the inner binary variables are coordinatewise separable, this minimax problem is equivalently the single-level nonsmooth minimization problem
\begin{equation}\label{sec5.3-2}
\min_x \frac12\left\{\sum_{i\notin\mathcal I}\left((a_i^\top x)^2-\widetilde b_i\right)^2
+\sum_{i\in\mathcal I}\max\left[(a_i^\top x)^4,
\left((a_i^\top x)^2-1\right)^2\right]\right\}.
\end{equation}

We partition the training samples into three disjoint groups according to the status of their labels: samples with unknown labels in \(\mathcal I\), samples with observed but noisy labels satisfying \(i\notin\mathcal I\) and \(\widetilde b_i\neq (b_{\rm true})_i\), and samples with observed clean labels satisfying \(i\notin\mathcal I\) and \(\widetilde b_i=(b_{\rm true})_i\).
The main experimental factor is the proportion of samples with unknown labels, while the observed noisy-label group represents additional label errors among the samples whose labels are observed.
This presentation directly reflects the information used by the two methods: discard-set LS removes the unknown group and fits the remaining observed labels, whereas BRLS keeps the unknown group and protects against binary label choices within it. The experiment uses signal dimension $m=200$, $r=10000$ training samples, and an independent test set $\{a_i^{\rm test}, (b_{\rm test})_i\}_{i=1}^{r_{\rm test}}$ with $r_{\rm test} = 30000$ samples.
To avoid class imbalance effects in the label-accuracy metric, we generate the training and test sets so that the clean binary labels are approximately balanced before introducing missing and noisy labels.
We report medians over ten independent trials.

We evaluate two quantities. The signal recovery error  ``$x$-err.'' is defined as
\[
{ x\text{-err.}} =
\frac{\min\{\|x-x_{\rm true}\|,\|x+x_{\rm true}\|\}}{\|x_{\rm true}\|},
\]
where the minimum accounts for the sign ambiguity of phase retrieval. The test label accuracy ``$b$-acc.'' is defined as
$$
b\text{-acc.} = 1- \frac{\sum_{i=1}^{r_{\rm test}} |\mathbf{1}\{((a_i^{\rm test})^\top x)^2 \ge 1/2\} - (b_{\rm test})_i| }{r_{\rm test}}.
$$

\begin{table}[h]
\centering
\caption{{Compare $x$-err. and $b$-acc. defined by LS solution of  (\ref{sec5.3-1}) with $C = \mathbf 0$ and BRLS solution of (\ref{sec5.3-2}). Unknown, observed noisy, and observed clean denote the label status of the training samples. }}
\label{tab:thresholded_phase_retrieval}

\begingroup
\setlength{\tabcolsep}{3.5pt}
\renewcommand{\arraystretch}{1.12}

\begin{tabular}{@{}ccc cc cc@{}}
\toprule
Unknown & Observed noisy & Observed clean & \multicolumn{2}{c}{$x$-err.} & \multicolumn{2}{c}{$b$-acc.} \\
\cmidrule(lr){4-5} \cmidrule(l){6-7}
& & & LS & BRLS & LS & BRLS \\
\midrule
$20\%$ & $5\%$ & $75\%$ & 0.0323 & \textbf{0.0178} & 0.9108 & \textbf{0.9495} \\
$30\%$ & $5\%$ & $65\%$ & 0.0315 & \textbf{0.0146} & 0.9136 & \textbf{0.9598} \\
$40\%$ & $5\%$ & $55\%$ & 0.0301 & \textbf{0.0129} & 0.9168 & \textbf{0.9639} \\
$50\%$ & $5\%$ & $45\%$ & 0.0729 & \textbf{0.0186} & 0.8510 & \textbf{0.9545} \\
$60\%$ & $5\%$ & $35\%$ & 0.0918 & \textbf{0.0389} & 0.8349 & \textbf{0.9067} \\
\midrule
$20\%$ & $10\%$ & $70\%$ & 0.0339 & \textbf{0.0178} & 0.9053 & \textbf{0.9497} \\
$30\%$ & $10\%$ & $60\%$ & 0.0332 & \textbf{0.0153} & 0.9072 & \textbf{0.9569} \\
$40\%$ & $10\%$ & $50\%$ & 0.0323 & \textbf{0.0130} & 0.9094 & \textbf{0.9635} \\
$50\%$ & $10\%$ & $40\%$ & 0.1253 & \textbf{0.0286} & 0.8085 & \textbf{0.9317} \\
$60\%$ & $10\%$ & $30\%$ & 0.1619 & \textbf{0.0493} & 0.7829 & \textbf{0.8785} \\
\bottomrule
\end{tabular}

\endgroup
\end{table}

Table~\ref{tab:thresholded_phase_retrieval} shows that BRLS improves signal recovery and clean test-label prediction in most regimes tested. As the unknown group becomes large, the signal recovery error and test label accuracy of the LS and BRLS methods
increase and decrease, respectively. This is expected because the effective amount of known information is necessary for stable recovery in this problem.
}

\section{Conclusion}\label{sec:conclusions}
This paper proposes the binary robust least squares (\ref{eq:binary_RLS}) model, which unifies a broad class of robust least squares problems including those with binary label uncertainty and hypercube-constrained adversarial noise.
{We present theoretical results for the
inner binary maximization problem with a convex quadratic objective function
$\frac{1}{2}\|F(x)-Cy\|^2$ and  complexity bounds for finding  minimaxi points of  problem (\ref{eq:binary_RLS}) by some existing algorithms based on
the structure of the matrix $C$ and inner oracle for supermodular/submodular functions. In particular, the projected gradient algorithm for
acute $C$, the double greedy algorithm for obtuse $C$ and the SDP approximation solvers for general $C$ and affine $F$.  This oracle-based viewpoint yields explicit global minimax, approximate minimax, and Moreau-envelope stationarity guarantees across the linear and nonlinear settings. }
{Future work may extend the framework to nonlinear submodular and nonlinear general-$C$ settings and structure-exploiting algorithms for mixed-angle noise propagation matrices beyond the general SDP-based approximation oracle.}

\section*{Acknowledgements}
{This work is supported by Natural Science Foundation of Shandong Province (No. ZR2025MS24) and National Science Foundation of China (No. 12371099), CAS-Croucher Funding Scheme for the CAS AMSS-PolyU Joint Laboratory in Applied Mathematics and Hong Kong Research Grant Council projects PolyU15300024, JLFS/P-501/24.
{The authors would like to thank the three anonymous reviewers for their insightful and constructive comments which help us to improve the quality of this paper.}

	\bibliographystyle{siam}
	\bibliography{RNLS_bib}	
\end{document}